\documentclass[11pt]{amsart}
\usepackage{times}
\usepackage[T1]{fontenc}
\usepackage{amssymb, amsthm, amsmath}
\usepackage{mathrsfs}

\usepackage[active]{srcltx}

\usepackage{todonotes}

\usepackage{setspace}
\usepackage{geometry}

\usepackage{hyperref}
\DeclareMathOperator{\dM}{DM}
\DeclareMathOperator{\acl}{acl}
\DeclareMathOperator{\dcl}{dcl} 
 
 \DeclareMathOperator{\id}{id}
 \DeclareMathOperator{\fr}{Fr}

 \DeclareMathOperator{\dom}{dom}

\DeclareMathOperator{\tp}{tp}

\DeclareMathOperator{\mr}{RM}
\DeclareMathOperator{\cl}{cl}

\newtheorem{introtheorem}{Theorem}

\newtheorem{theorem}{Theorem}[section]

\newtheorem{claim}{Claim}[theorem]

\newtheorem{corollary}[theorem]{Corollary}

\newtheorem{fact}[theorem]{Fact}
\newtheorem{lemma}[theorem]{Lemma}

\newtheorem{proposition}[theorem]{Proposition}

\newtheorem*{assumption}{Assumption}

\newcommand{\efi}{\hyperlink{Aefbi}{(EfI)}}
\newtheorem*{efbi}{\fbox{{\large A}} \hypertarget{Aefbi}{EfI}}

\newcommand{\gendif}{\hyperlink{Agen-dif}{(Gen-Dif)}}
\newtheorem*{gen-dif}{\fbox{{\large A}} \hypertarget{Agen-dif}{Gen-Dif}}

\newcommand{\VJP}{\hyperlink{Avjp}{(VJP)}}
\newtheorem*{vjp}{\fbox{{\large A}} \hypertarget{Avjp}{VJP}}

\newcommand{\MVTay}{\hyperlink{Amvtay}{(MVTay)}}
\newtheorem*{mvtay}{\fbox{{\large A}} \hypertarget{Amvtay}{MVTay}}

\newcommand{\BCen}{\hyperlink{Abcen}{(B-Cen)}}
\newtheorem*{bcen}{\fbox{{\large A}} \hypertarget{Abcen}{B-Cen}}

\newcommand{\minballn}{\hyperlink{Amin-ball}{(Cballs)}}
\newtheorem*{min-balln}{\fbox{{\large A}} \hypertarget{Amin-ball}{Cballs}}

\theoremstyle{definition}
\newtheorem{definition}[theorem]{Definition}
\newtheorem{example}[theorem]{Example}
\newtheorem{remark}[theorem]{Remark}

\newtheorem{notation}[theorem]{Notation}

\newcommand{\Cc}{{\mathbb{C}}}
\newcommand{\Rr}{{\mathbb{R}}}
\newcommand{\Nn}{{\mathbb{N}}}
\newcommand{\Qq}{{\mathbb{Q}}}

\newcommand{\m}{\textbf{m}}
\newcommand{\bk}{\textbf{k}}

\newcommand{\CD}{{\mathcal D}}
\newcommand{\CL}{{\mathcal L}}
\newcommand{\CK}{{\mathcal K}}
\newcommand{\CN}{{\mathcal N}}

\newcommand{\CM}{{\mathcal M}}

\newcommand{\CC}{{\mathcal C}}

\newcommand{\CO}{{\mathcal O}}

\newcommand{\CF}{{\mathscr F}}
\newcommand{\CG}{{\mathcal G}}

\newcommand{\0}{\emptyset}

\renewcommand{\phi}{\varphi}

\newenvironment{claimproof}[1][\proofname]
  {%
    \proof[#1]%
  }
  {%
    \endproof%
  }

\newcommand{\dclindep}[1][]{%
  \mathrel{
    \mathop{
      \vcenter{
        \hbox{\oalign{\noalign{\kern-.3ex}\hfil$\vert$\hfil\cr
              \noalign{\kern-.7ex}
              $\smile$\cr\noalign{\kern-.3ex}}}
      }
    }\displaylimits_{#1}
  }
}

\long\def\symbolfootnote[#1]#2{\begingroup%
\def\thefootnote{\fnsymbol{footnote}}\footnote[#1]{#2}\endgroup}

\makeatletter

\def\Ind#1#2{#1\setbox0=\hbox{$#1x$}\kern\wd0\hbox to 0pt{\hss$#1\mid$\hss}
\lower.9\ht0\hbox to 0pt{\hss$#1\smile$\hss}\kern\wd0}

\def\Notind#1#2{#1\setbox0=\hbox{$#1x$}\kern\wd0\hbox to 0pt{\mathchardef
\nn=12854\hss$#1\nn$\kern1.4\wd0\hss}\hbox to
0pt{\hss$#1\mid$\hss}\lower.9\ht0 \hbox to
0pt{\hss$#1\smile$\hss}\kern\wd0}

\def\la{\langle}
\def\ra{\rangle}
\def\qp{\mathbb Q_p}

\def\oD{\overline{D}}
\def\dpr{\mathrm{dp\text{-}rk}}
\def\sub{\subseteq}

\makeatother

\title{Interpretable Fields in Various Valued Fields}

\date{\today}
\author{Yatir Halevi}
\address{The Fields Institute for Research in Mathematical Sciences, Toronto, Canada}
\email{ybenarih@fields.utoronto.ca}

\author{Assaf Hasson}
\address{Department of Mathematics, Ben Gurion University of the Negev, Be'er-Sheva 84105, Israel}
\email{hassonas@math.bgu.ac.il}

\author{Ya'acov Peterzil}
\address{Department of Mathematics, University of Haifa, Haifa, Israel}
\email{kobi@math.haifa.ac.il}

\thanks{The first author was supported by the Kreitman foundation fellowship and the Fields Institute for Research in Mathematical Sciences. The third author was partially supported by Israel Science Foundation grant number 290/19 }

\begin{document}

\begin{abstract}
Let $\CK=(K,v,\ldots)$ be a dp-minimal expansion of a non-trivially valued field of characteristic $0$ and  $\CF$  an infinite field interpretable in $\CK$.

Assume that $\CK$ is one of the following: (i)  $V$-minimal,  (ii) power bounded $T$-convex,  or (iii) $P$-minimal (assuming additionally in (iii) generic differentiability of definable functions). Then $\CF$ is definably isomorphic to a finite extension  $K$ or, in cases (i) and (ii), its residue field.  In particular, every infinite field interpretable in $\Qq_p$ is definably isomorphic to a finite extension of $\Qq_p$, answering a question of Pillay's.

Using Johnson's work on dp-minimal fields and the machinery developed here, we conclude that if $\CK$ is an infinite dp-minimal pure field then every field definable in $\CK$ is definably isomorphic to a finite extension of $K$.

The proof avoids elimination of imaginaries in $\CK$ replacing it with a reduction of the problem to certain distinguished quotients of $K$.
\end{abstract}

\maketitle

\tableofcontents

\section{Introduction}

We consider families of valued fields of characteristic $0$, such as algebraically closed valued fields, real closed valued fields, or p-adically closed fields. The goal of this work is to classify infinite fields interpretable in certain dp-minimal expansions of such a field $K$, namely fields whose universe is given as a quotient $X/E$ of a definable set $X\sub K^n$ by a definable equivalence relation $E$.

The main result of our paper is:

\begin{introtheorem}\label{T: main interp}[Theorem \ref{T: main thoerem}]
	Let $\CK=(K,v, \dots)$ be a dp-minimal valued field with residue field $\bk$ and let $\CF$ be an infinite field interpretable in $\CK$. Then:
	%and value group $\Gamma$.
%	Let $\CF$ be an infinite field interpretable in $K$. Then:
	\begin{enumerate}
\item  If $\CK$ is a $V$-minimal valued field then $\CF$ is definably isomorphic to $K$ or $\textbf{k}$. 
\item If $\CK$ is a power bounded $T$-convex valued field then $\CF$ is definably isomorphic to one of $K$, $K(\sqrt{-1})$, $\textbf{k}$ or $\textbf{k}(\sqrt{-1})$. In particular, the result holds if $\CK$ is a real closed valued field.
		\item If $\CK=(K,v,\dots)$ is a P-minimal valued field with the property that any definable function is differentiable outside of nowhere dense subset of its domain, then $\CF$ is definably isomorphic to a finite extension of $K$.  In particular, the result holds if $\CK$ is $p$-adically closed.
		
% In particular, the results holds if $\CK$ is an algebraically closed valued of residue characteristic $0$.
	\end{enumerate}
\end{introtheorem}

This type of  classification of definable quotients originates in Poizat's model theoretic consideration  of the Borel-Tits theorem, \cite{PoiFields}.
In the setting of (pure) algebraically closed valued fields, a study of interepretable groups and fields  was carried out by Hrushovski and Rideau-Kikuchi in \cite{HrRid}. The result for V-minimal fields generalizes the characteristic $0$ case of \cite[Theorem 6.23]{HrRid}.

In \cite{PilQp} Pillay showed that every infinite field definable in $\mathbb Q_p$ is definably isomorphic to a finite extension of $\mathbb Q_p$, and asked whether the same is true for interpretable fields. The above solution to his question was part of the motivation for this work.\footnote{A positive answer to Pillay’s question on interpretable fields in $\qp$ was announced also by E. Alouf, A. Fornasiero and J. de la Nuez Gonzalez.}

Theorem \ref{T: main interp} classifies in particular fields  \emph{definable} in $\CK$. In that regard it generalises the analogous result of Bays and the third author for  real closed valued fields, \cite{BaysPet}. \\

Using Johnson's theorem on dp-minimal fields (see Fact \ref{F:Johnson DP} below), and the machinery developed here, we also prove:

\begin{introtheorem}\label{T:main def}[Corollary \ref{dp minimal fields}]
	Let $\CK=(K,+,\cdot)$ be a pure dp-minimal field of characteristic $0$. Then every field definable in $\CK$ is definably isomorphic to a finite extension of $K$.
\end{introtheorem}

Note that without the purity assumption the result fails, in the case of strongly minimal fields, \cite{Hr2}. In the o-minimal case, however, purity is superfluous, as shown by Otero, Pillay and the third author, \cite{OtPePi}. In fact, in the notation of Theorem 2, if $\CK$ is unstable then both the purity and characteristic assumptions can be replaced by the requirement that definable functions are generically differentiable (Proposition \ref{prop-main}).
\subsection{Strategy of the proof and structure of the paper}

The study of imaginaries in algebraically closed valued fields was initiated by Holly in \cite{HolEOI1} and first studied in depth and in full generality, in the setting of algebraically closed valued fields, by Haskell, Hrushovski and Macpherson in \cite{HaHrMac1} and \cite{HaHrMac2}.
Similar results were later obtained by Mellor for real closed valued fields, \cite{MelRCVFEOI},  and by Hrushovski, Martin and Rideau-Kikuchi for p-adically closed fields \cite{HrMarRid}.
The work on interpretable fields in \cite{HrRid} uses these theorems on elimination of imaginaries, and accomplishes  considerably more than the classification of such fields.

%Elimination of imaginaries in real closed valued fields analogous to that of \cite{HaHrMac1} was proved by Mellor in \cite{Mellor}. Some results analogous %to  \cite{HaHrMac2} for real closed valued fields were obtained by Ealy, Haskell and Marikova in  \cite{EHM}.

 We adopt a different approach circumventing elimination of imaginaries, and avoiding almost completely the so called \emph{geometric sorts}.
%Our proof is direct and does not make use of the above mentioned works on elimination of imaginaries in valued fields.
In fact, our main result covers  expansions of valued fields by analytic functions where no elimination of imaginaries results are currently available  (see \cite{HaHrMac3}). Our proof is based on the analysis of  dp-minimal subsets of the interpreted field $\CF$, and as such it borrows ideas from Johnson's work on fields of finite dp-rank (see for example  \cite{JohnDpFin1A} and \cite{JohnDpFin1B}), as well as Otero-Peterzil-Pillay \cite{OtPePi}.

The general setting in which we carry out most our work is that of dp-minimal uniformities, suggested and axiomatized by
Simon and Walsberg in \cite{SimWal} (called here {\em SW-uniformities}). Their work yields some sort of cell decomposition, generic continuity of definable functions and other properties similar to those of o-minimal structures. We discuss and expand their results  in Section  \ref{S:SW-uniformities}.
Another general setting which fits all three cases in our main theorem is that of \emph{1-h-minimal valued fields}
as developed, in characteristic $0$, in \cite{hensel-min,hensel-minII} by Cluckers, Halupczok, Rideau-Kikuchi and Vermeulen. 

In Section \ref{S:def fields}, under the additional assumption that definable functions in $\CK=(K,v\dots)$ are generically differentiable we use ideas of \cite{MarikovaGps} and \cite{OtPePi} in the o-minimal setting to show that if $\CF$ is an infinite field \emph{definable} in $\CK$ then $\CF$ can be definably embedded into a subring of $M_n(K)$, the ring of $n\times n$ matrices over $K$, for some $n\in \Nn$. One of the main technical results of our work shows that in fact, to obtain the same result it suffices that $\CF$ is \emph{locally strongly internal to $K$}, namely, that there is an infinite subset of $\CF$  in definable bijection with a subset of $K^m$ (some $m$).

This observation allows us, in the cases covered by Theorem \ref{T: main interp}, to circumvent elimination of imaginaries and generalise the result to fields interpretable in $\CK$. In Section \ref{S:the distinguished sorts} we prove, using a reduction to unary imaginaries, that any interpretable field is (almost) locally strongly internal to one of four distinguished sorts: the valued field, $K$, the value group, $\Gamma$, the residue field, $\bk$, or the closed balls of radius $0$, $K/\CO$. Our results on definable fields are used to obtain the main theorem when  $\CF$ is locally strongly internal to $K$ or to $\bk$. Under the assumptions of Theorem \ref{T: main interp}, local strong internality to $\Gamma$ can be eliminated by known results from o-minimality (the $P$-minimal case being simpler). Section \ref{S:K/O} is dedicated to eliminating the case of $K/\CO$.  In the final section of the paper we combine all the results obtained in this work to prove Theorem \ref{T: main interp}.\\

\noindent\textbf{Remark} The current article generalizes and replaces two previous preprints, \cite{HaPetRCVF}, \cite{HAHaPe}. \\

\noindent\textbf{Acknowledgments}
We thank Immanuel Halupczok, Ehud Hrushovski, Itay Kaplan and Dugald Macpherson for several discussions during the work on the article.
We thank Pablo Cubides Kovacsics and the model theory group in D{\"u}ssseldorf for their careful reading of previous preprints.

\section{Notation and preliminaries}

\subsection{Model theory}
We use standard model theoretic notation see e.g. \cite{TZ,SiBook}. Lower case  Latin letters $a,b,c$ are used to denote elements and  tuples, capital letters $A,B,C$ are used for sets. Abusing  notation we write $a\in A$ for tuples when the length of the tuple is immaterial or understood from the context.

We use $\mathcal M,\mathcal K$, etc. to denote structures whose universe is $M$, $K$, respectively.
%\footnote{An exception for this will be $\CO$, which will be defined in Section \ref{ss:valued field}}. 
We allow multi-sorted structures but from Section \ref{S:the distinguished sorts} on we focus on one-sorted structures expanding a value field,  where all other sorts are  taken from $\mathcal{M}^{eq}$ (i.e. $\mathcal{M}$ together with a sort for every quotient $M^n/E$, where $E$ is a $\emptyset$-definable equivalence relation, and a symbol for the projection $M^n\to M^n/E$). In such a setting by {\em an interpretable set}, we  mean a set definable in $\mathcal{M}^{eq}$, whereas by a definable set we mean a definable subset of $M^n$ for some $n$ (possibly with parameters). We may also refer to definable subsets in $S$  for some (imaginary) sort $S$ of $\CM$, by which we mean a subset of $S^n$ (for some $n$) definable in $\CM$, possibly with parameters. 

%A structure $\mathcal{M}$ is %\emph{$\kappa$-saturated}, if every type $p$ %over $A$ with $|A|<\kappa$ is realized in %$\mathcal{M}$
%The structure $\mathcal{M}$ is saturated if it is $|\mathcal{M}|$-saturated. 
%\todo{I omitted "saturated"}

A \emph{monster model} for a first order theory $T$ is a large sufficiently saturated and sufficiently homogeneous model containing all sets and models (as elementary substructures) we will encounter. All subsets and models are \emph{small}, i.e. of cardinality smaller than the saturation level of the monster. When $\CM$ is $\kappa$-saturated subsets $A\sub M$ are always assumed to be of cardinality $<\kappa$.

Throughout the paper we apply consequences of dp-minimality mostly as a black box. However, since the actual definition is used in a couple of places,  we remind: 
\begin{definition}
Let $T$ be a complete theory with some monster model $\mathbb{U}$, $P$ a partial type over a set $A\subseteq \mathbb{U}$ and  $\kappa$  a cardinal. We say that $\dpr(P)<\kappa$ if for every family $\la I_t: t<\kappa\ra$ of $A$-mutually indiscernible sequences and $b\models P$, there is $t<\kappa$ such that $I_t$ is indiscernible over $Ab$. We say that $\dpr(P)=\kappa$ if $\dpr(P)<\kappa^+$ but not $\dpr(P)<\kappa$.
For any $a\in \mathbb{U}$ and small set $A$, $\dpr(a/A)$ is defined to be $\dpr(\tp(a/A))$.

A (one-sorted) structure has \emph{NIP} if $\dpr(x=x)<|T|^+$, it is \emph{dp-minimal} if $\dpr(x=x)=1$ and \emph{dp-finite} (or \emph{of finite dp-rank}) if $\dpr(x=x)=n$ for some $n\in \Nn$. See \cite[Section 4.2]{SiBook}.
\end{definition}

% A multi-sorted structure $\CM$, with a distinguished main sort (e.g., a valued field) is dp-minimal if the main sort is. In that case consequences of dp-minimality are applied directly only to (type) definable sets in the main sort. 
It follows from sub-additivity of dp-rank below that any set interpretable in a dp-minimal sort has finite dp-rank.  We will use this fact throughout the paper without further reference. 

We refer the reader to \cite[Section 4]{SiBook} for the basic properties of dp-rank out which we  emphasize sub-additivity: 

\begin{fact}\cite{KaOnUs} 
For every $a,b\in \CM$ and small set $A$
\[\dpr(a,b/A)\leq \dpr(a/bA)+\dpr(b/A).\]
\end{fact}

A consequence of dp-minimality is that dp-rank is local (\cite[Theorem 0.3(2)]{Simdp}) in the sense that $\dpr(a/A)=\min\{\dpr(\phi): \phi\in \tp(a/A)\}$. We make use of this fact at some points.

\begin{definition}
For a partial type $P$ over $A$ and some $A$-definable function $f$, with $P\vdash \dom(f)$, we let $f_*(P)$ be the partial type $\{f(X): X\in P\}$.
\end{definition}

The choice of working with dp-rank and dp-minimal structures permits us to study  a rich variety of examples. It is also influenced by the foundational result of Johnson on arbitrary dp-minimal fields. While most of our work does not rely on his theorem (with the exception of Corollary \ref{dp minimal fields}), we refer to it several times in this work:
\begin{fact}[\cite{JohnDpJML}]\label{F:Johnson DP}
    Let $\CK=(K,+,\cdot,\dots) $ be a dp-minimal field. Then $K$ is either algebraically closed, real closed or $\CK$ admits a definable henselian valation.
\end{fact}

\subsection{Valued fields}\label{ss:valued field}
A valued field $(K,v)$ is a field $K$ together with a surjective group homomorphims $v:K^\times\to \Gamma$, where $\Gamma$ is an ordered abelian group (we set $v(0)=\infty$, where $\infty$ has the obvious properties), satisfying the non-archimedean inequality:
\[v(x+y)\geq \min\{v(x),v(y)\},\]
for any $x,y\in K$. The group $\Gamma$ is the \emph{value group} of the valued field $K$. 

For $\gamma\in \Gamma$ and $a\in K$ we let
\[B_{>\gamma}(a)=\{x\in K:v(x-a)>\gamma\}\] and
\[B_{\geq \gamma}(a)=\{x\in K:v(x-a)\geq \gamma\}\]
 be the open and closed balls  of radius $\gamma$ centered at $a$, respectively. For us, a ball is always infinite, i.e. $\gamma\neq \infty$.
 
The closed ball $B_{\geq 0}(0)$ is a ring called the \emph{valuation ring of $K$},  which we denote here by $\mathcal{O}$. It is a local ring with maximal ideal $\m:=B_{>0}(0)$. The quotient $\bk:=\CO/\m$ is the  \emph{residue field}.  We refer to \cite{EnPr} for the basics of valuation theory.

Some texts on valuation theory use multiplicative notation (and reverse ordering) for $\Gamma$, and $|x|$ instead of $v(x)$. Thus for example, $\CO$ becomes $\{x\in K:|x|\leq 1\}$ in this notation. We will comment on  the multiplicative writing when we find it more intuitive (e.g. when we discuss Taylor approximations of functions).

Throughout, $\CK=(K,+,\cdot,\dots)$ will denote a dp-minimal expansion of an infinite field $K$. If $\CK$ expands a valued field we let $\Gamma$ and $\bk$ denote  the imaginary sorts whose universes are the value group  and the residue field respectively.

The imaginary sort $K/\CO$ of all closed balls of radius $0$ will play an important role in Section \ref{S:K/O}.
For ease of reference we use the following \emph{ad hoc} definition: 
\begin{definition}
    Given a structure $\mathcal K$ expanding a valued field $K$,  the \emph{distinguished sorts of $\CK$} are $K$, $\Gamma$, $\bk$ and $K/\CO$. 
\end{definition}

We also refer at several points to the  sort $\mathrm{RV}=K^\times/(1+\m)$ and the associated exact sequence of abelian groups:
\[1\to \bk^\times\to \mathrm{RV}\to \Gamma\to 0.\]

\subsection{The main examples}\label{ss:main examples}
For the reader's convenience we remind the definitions of those theories of valued fields appearing in the statement of Theorem \ref{T: main interp}.

\subsubsection{P-minimal valued fields}\label{sss:p-minimal}
 The notion of $P$-minimality was introduced by Haskell and Macpherson, \cite{HasMac}:
\begin{definition}
An expansion of a $p$-adically closed valued field  $\CK:=(K,v,\dots)$ is \emph{P-minimal} if $\Gamma$ is a $\mathbb{Z}$-group and in every structure $\CL\equiv \CK$  every definable subset of $L$ is quantifier-free definable in the Macintyre language of valued fields. I.e., every such set can be written as a disjunction of sets of the form
\[
\{x\in L: \gamma_1< v(x-a)<\gamma_2 \wedge P_n(\lambda\cdot (x-a))\},
\]
where $P_n$ is the $n$-th power predicate, $\lambda\in L$ and $\gamma_1,\gamma_2\in \Gamma_L\cup\{\infty,-\infty\}$, $n\in\mathbb{N}$ and $a\in L$.
\end{definition}

By \cite[Section 6]{DoGoLi}, P-minimal  valued fields are dp-minimal.

\begin{remark}
 %\begin{enumerate}
     It follows directly that every definable subset of $\Gamma^n$, $n\in \mathbb N$,  in a $P$-minimal field is definable in Presburger Arithmetic, and that any such bounded subset of $\Gamma$ has an infimum. 
    %\item There is also a slightly weaker definition for $P$-minimality first introduced by Cluckers in \cite[Definition 5]{ClucPresburger} (dropping the requirement that $\Gamma$ is a $\mathbb{Z}$-group). If, using Cluckers' definition, every P-minimal valued field is dp-minimal (as is probably the case) then the results of this paper will also work for P-minimal fields using this weaker definition.
 %\end{enumerate}
 \end{remark}

\subsubsection{C-minimal valued fields}\label{sss:c-minimal}

Let $(K,v)$ be a valued field.
\begin{definition}
A \emph{Swiss cheese} is a set of the form $b\setminus \bigcup_i b_i$, where $b$ is a ball and the $b_i$ are finitely many subballs of $b$ (where all balls may be either closed or open).

A Swiss cheese $b\setminus \bigcup_i b_i$ is \emph{nested} in another Swiss cheese $d\setminus \bigcup_i d_i$ if there exists  $i$ such that $b=d_i$.
\end{definition}

% \begin{remark}\label{R:swiss}
% Note that if $S_1, \dots S_k$ are pairwise disjoint Swiss cheeses none of which is nested in another then for any Swiss cheese $S\sub \bigcup S_i$ we have $S\sub S_i$ for some $i$. 
% \end{remark}

\begin{fact}[Holly, \cite{holly}]\label{F:holly-swiss}
Any definable subset of an ACVF has a unique decomposition as a finite disjoint union of non-nested Swiss cheeses.
\end{fact}

\begin{definition}
An expansion  $\CK=(K,v,\dots)$ of a valued field is \emph{C-minimal} if in every  $\CK'\equiv \CK$, every definable subset of $K'$ is a finite boolean combination of balls.\footnote{Macpherson and Steinhorn's original definition of a C-minimal field is quite different, but the two definitions coincide for valued fields, see \cite{MacSte}.}
\end{definition}

It follows that every definable $X\sub K$
 has a unique decomposition as a finite disjoint union of non-nested Swiss cheeses.
%They conclude that every such field admits a definable valuation and then these two definitions coincide. See \cite{MacSte}.}.

Haskel and Macpherson showed, \cite{HasMacCell}, that every C-minimal valued field is an algebraically closed valued field. In addition, $\Gamma$ is o-minimal and $\bk$ is strongly minimal so in particular, the residue field is algebraically closed, and the value group is divisible.  By \cite[Corollary 4.3]{DoGoLi}, any C-minimal field is dp-minimal.

\subsubsection{V-minimal valued fields}\label{sss:v-minimal}
Although we only use the following  as a black box, we  introduce here the notion of V-minimality, from Hrushovoski-Kazhdan's \cite{HrKa}:

\begin{definition} 
A {\em C-minimal} theory expanding ACVF$_{0,0}$ is {\em V-minimal} if  for every $\CK=(K,v,\dots)\models T$,

(1) Every definable relation on RV$^n$, $n\in \mathbb N$,   is definable in the language of  valued fields. 

(2) For every definable \textbf{chain} of closed balls $W$, we have $\bigcap W\neq \0.$

(3) For every $A\sub K$, if $Y$ is an $A$-definable finite set of closed balls then for each $b\in Y$, $\acl(A)\cap b\neq \0$.
\end{definition} 

\subsubsection{$T$-convex valued fields}\label{sss:t-convex}
We now recall the notion of a $T$-convex theory, due to van den Dries and Lewenberg, \cite{vdDriesLewen}, \cite{vdDries-Tconvex}.

\begin{definition}\label{D:t-convex}
Let $T$ be an o-minimal expansion of a real closed field in a signature $\CL$. A \emph{T-convex valued field} is an expansion of $T$ by a predicate for a convex valuation ring $\CO$ that is closed under all $\0$-$\CL$-definable continuous functions.

The resulting theory, $T_{conv}$, is called {\em power bounded} if $T$ is, namely if every $\CL$-definable function from $(0,\infty)$ to $K$ is eventually bounded by a definable $\CL$-automorphism of $K^{>0}$ (such automorphisms are called {\em power functions}).
\end{definition}

Recall that a linearly ordered structure $\CM=(M,<,\cdots)$ is \emph{weakly o-minimal} if every definable subset of $M$ is a finite union of convex sets.
By \cite[Corollary 3.3]{DoGoLi}, weakly o-minimal structures are dp-minimal.

T-convex power-bounded valued fields satisfy the various properties used here:

\begin{fact}\label{F:facts on t-convex}
\begin{enumerate}
    \item Every T-convex valued field is weakly o-minimal \cite[Corollary 3.14]{vdDriesLewen}.
    \item The residue field of any T-convex valued field is o-minimal \cite[Theorem A]{vdDries-Tconvex}.
    \item The value group of any T-convex power-bounded valued field is a pure ordered vector space over the ordered field of powers of $K$ \cite[Theorem B]{vdDries-Tconvex},  hence it is o-minimal. See also \cite[Proposition 4.3]{vdDries-Tconvex}.
\end{enumerate}
\end{fact}

\section{Simon-Walsberg uniformities and their properties}\label{S:SW-uniformities}

\subsection{Preliminaries on definable uniform structures}
One of the challenges of this article was to find a proper framework which fits a variety of dp-minimal expansions of valued fields. Simon and Walsberg in  \cite{SimWal} (and in somewhat greater generality  Dolich and Goodrick, \cite{DolGodViscerality})
provide an elegant setting of dp-minimal uniformities which suits well our needs. We present here the basic definitions and properties and develop further some local properties of such uniformities.

\begin{assumption}
Throughout \underline{entire} Section $3$ we assume that $\CM$ is $|T|^+$-saturated, where $T=\mathrm{Th}(\CM)$.
\end{assumption}

\begin{definition}
A definable set $D$ in a (possibly multi-sorted) structure $\CM$ has a \emph{definable uniform structure}, or \emph{a definable uniformity}, if there is a formula $\theta(x,y, z)$ such that $\mathscr{B}=\{\theta(D^2, t):t\in T\}$ satisfies the following:
    \begin{enumerate}
    \item the intersection of the elements of $\mathscr{B}$ is $\{(x,x): x\in D\}$;
    \item if $U\in \mathscr{B}$ and $(x,y)\in U$ then $(y,x)\in U$;
    \item for all $U,V\in \mathscr{B}$ there is a $W\in \mathscr{B}$ such that $W\subseteq U\cap V$;
    \item for all $U\in \mathscr{B}$ there exists $V\in \mathscr{B}$ such that
    \[\{(x,z)\in D^2: (\exists y\in D)\left( (x,y)\in V, (y,z)\in V\right)\}\subseteq U.\]
\end{enumerate}
\end{definition}
A definable uniform structure  induces a definable topology on $D$ whose basic open sets are $U[x]:=\{y: (x,y)\in U\}$, as $U$ ranges over $\mathscr B$.  

For the purposes of the present work the main source of examples of such structures are expansions of definable (abelian) groups equipped with a definable neighbourhood basis $\mathscr B_0$ at $0$ and the associated uniformity
$$\mathscr B=\{(x,y)\in G: xy^{-1}\in U\}_{U\in \mathscr B_0}.$$   
% Throughout this section, unless explicitly stated otherwise, a uniformity on an ordered group is always the one given by the order topology, and in a valued field $(K,v)$ the uniformity is given via the group $(K,+)$ and the valuation topology.
%\todo{Kobi: why do we need this assumption? It saves few words for $K/O$, but we do not use it mathematically? Yatir: I agree - It was a leftover. I commented it out.}
%\todo{Do  we also assume that in an SW field multiplication and inverse are continuous?}

%\todo{ I think we don't even use the fact that $+$ is continuous}

%obtained when $D$ expands the structure of a group together with a definable neighbourhood basis of the identity and when $D$ is a valued field.

Based on \cite{SimWal}, we define:
\begin{definition} A set $D$, equipped with a definable uniformity, is called \emph{an SW-uniform structure} (or {\em an SW-uniformity})
if:
\begin{itemize}
    \item $D$ is dp-minimal,
    \item $D$ has no isolated points,
    \item every infinite definable subset of $D$ has nonempty interior.
\end{itemize}
\end{definition}
Throughout this section, whenever we say that $D$ is an SW-uniformity, we tacitly assume that $D$ is a definable set in some ambient (multi-sorted) structure. Note that any SW-uniformity is unstable.

The following examples emphasis the relevance of this setting to the present work: 
\begin{example}\label{E:SW-uniformities}
\begin{enumerate}
    \item A dp-minimal expansion of a divisible ordered abelian group is an SW-uniformity \cite{SimDPMinOrd}.
    \item In particular, every weakly o-minimal expansion of an ordered group is an SW-uniformity \cite[Theorem 5.1]{MacMaSt}.
    \item A dp-minimal expansion of a non-trivially valued field is an SW-uniformity. Indeed, it has no isolated points since the field topology is not trivial. Every infinite definable subset has non empty interior by, e.g., \cite[Proposition 3.6]{JaSiWa2015}.
    \item 
    Johnson in his work \cite{JohnDpJML} defines a topology on every dp-minimal expansion of a non strongly minimal field $K$. This topology, which he calls {\em the canonical topology}, has as a base the family of sets ${\bf B}=\{X-X:X\sub K \mbox{ definable and infinite}\}$, \cite[Theorem 6.5]{JohnDpJML}.  It  turns out to yield an SW-uniform structure,  see also \cite[Proposition 1.1]{SimWal}. The converse is true as well: Assume that  $\CK$ is a dp-minimal expansion of a topological field, with a definable basis for its topology $\tau$ making it into an SW-uniform structure. Then $\tau$ equals the canonical topology of Johnson.
    Indeed, since every infinite definable $X\sub K$ has a non-empty interior it follows that the family ${\bf B}$ forms  a base for $\tau$.
    \item In the sequel (Lemma \ref{L:K/O is SW}) we prove that if $\CK=(K,v,\dots)$ is a dp-minimal valued field with a dense value group then (under an additional technical assumption) $K/\CO$ admits the structure of an SW-uniformity.
\end{enumerate}
\end{example}

It follows from the work of Simon and Walsberg that, in SW-uniformities, the notion of dp-rank coincides with other notions of dimension.
For $X\sub D^n$, {\em the topological dimension of $X$} is defined to be the maximal $k\leq n$ such that some projection of $X$ onto $k$ of the coordinates contains an open set.
The {\em $\acl$-dimension} of a tuple $a\in M^n$  over $A\sub M$ (even when $\acl$ does not satisfy exchange) is the minimal  $k\leq n$ for which there exists a sub-tuple
$a'\sub a$, such that $a\in \acl(a'A)$. Then, the   {\em $\acl$-dimension}  of an $A$-definable set $X\sub M^n$ is defined to be the maximum of the $\acl$-dimension of $a/A$, for all $a\in X$.

By Theorem (\cite[Proposition 2.4]{SimWal}), if $D$ is an SW-uniformity in a $|T|^+$-saturated structure and $X\sub D^n$ is definable then $\dpr(X)$ equals the topological dimension and the $\acl$-dimension of $X$. The locality of dp-rank implies that the $\acl$-dimension of a tuple $a$ over $A$ equals $\dpr(a/A)$.
Using the definition of topological dimension we obtain definability of $\dpr$ in parameteres:

\begin{fact}\label{defofrank}\cite[Corollary 2.5]{SimWal} Let $X\sub D^{n+k}$ be $A$-definable. Then for every $k\in \mathbb N$, the set $\{a\in D^k:\dpr(X_a)=k\}$ is definable over $A$, where $X_a=\{b\in D^n:(b,a)\in X\}$.\end{fact}

 In addition,  definable functions in SW-uniformities are generically continuous: 
\begin{fact}\label{F:cont}\cite[Proposition 3.7]{SimWal}
    Let  $D$ be an $SW$-uniformity, $f:U\to W$ a definable function.  Then 
    \[
    C_f:=\{x\in U: f \text{ is continuous at } x \}
    \] 
    is definable and $\dpr(U\setminus C_f)<\dpr(U)$. 
\end{fact}
% \begin{proof}
%     The definability of $C_f$ follows from the definability of a basis for the topology. That $C_f$ is large is \cite[Proposition 3.7]{SimWal}.
% \end{proof}

We also need the following results:
\begin{fact} \cite[Lemma 4.6]{SimWal}\label{localhom} 
Let $X\sub D^n$ be a definable set in an SW-uniformity with $\dpr(X)=k$. Then there exists a definable subset $Y\subseteq X$ with $\dpr(Y)<k$ such that for every $a\in X\setminus Y$ there is a coordinate projection $\pi: X\to D^k$ that is a local homeomorphism at $a$. 
\end{fact}

For $X\sub Y\sub D^n$, the {\em relative interior of $X$ in $Y$}, $Int_Y(X)$, is the set of $a\in X$ such that for some open $V\ni a$ in $D^n$, 
$V\cap Y\sub X$ (hence $V\cap X=V\cap Y$).

\begin{fact}\cite[Corollary 4.4]{SimWal}\label{rel interior}
If $X\sub Y\sub D^n$ are definable in an SW-uniformity $D$ and $\dpr(X)=\dpr(Y)$  then $\dpr(Y\setminus Int_Y(X))<\dpr(Y)$.\end{fact}

\begin{remark}\label{R:product of one-dim in SW}
 It is easy to see that if $\CM$ is any dp-minimal structure then every infinite subset of $M^n$  has a definable subset of $\dpr=1$.

For an SW-uniformity $D$, if $X\subseteq D^n$ with $\dpr(X)=k$ then by Fact \ref{localhom},  there are definable open $X_1,\dots,X_k\subseteq D$ and a definable injection
of $X_1\times\dots\times X_k$ into $X$, possibly over additional parameters. If $X$ is defined over a model then so are the $X_i$ and the injection.
\end{remark}

%In the next section we show how the work of Simon and Walsberg can be extended to allow  local analysis (of definable sets) in %SW-uniformities. 

\subsection{Local analysis in SW-uniformities}

The  main goal of this section is to prove Proposition \ref{P:genos}, which generalizes a similar, useful, result from the theory of o-minimal structures. 
It says that 
given an $A$-definable set $X\sub D^n$ and $a\in X$ with $\dpr(a/A)=\dpr(X)$, 
there are arbitrarily small neighborhoods $V$ of $a$, defined over $B\supseteq A$, such that $\dpr(a/B)=\dpr(a/A)$.

If $\acl$ satisfies exchange then the  proposition would be quite easy to prove, as in the o-minimal setting. However, our eventual aim is to apply it in the SW-structure structure on $K/\CO$, where exchange fails.

The first step is due to Simon and Walsberg: 
%\todo{The below seems related to the basic properties of SW-uniformities so fits earlier}

\begin{lemma}\label{L:locally the same}
    Let $D$ be a definable SW-uniform structure.
    Let $X\subseteq Z\subseteq D^n$ be $\emptyset$-definable sets with $\dpr(X)=\dpr(Z)$. For every $d\in X$ if $\dpr(d)=\dpr(X)$ then there exists an open neighborhood $U\subseteq D^n$ of $d$ such that  $U\cap X=U\cap Z$.
\end{lemma}
\begin{proof}
     By Fact \ref{rel interior}, $\dpr(Z\setminus Int_Z(X))<\dpr(Z)$ and hence $d\in Int_Z(X)$.
\end{proof}

The following is probably well known.

\begin{lemma}\label{L:dp-rank bounded by externally definable}
Let $\CM$ be a structure of finite dp-rank,  $\CN\succ \CM$ an elementary extension and $X, Y$ definable over $\CM$ and $\CN$, respectively. If $X(\CM)\subseteq Y(\CM)$ then $\dpr(X)\leq \dpr(Y)$.
\end{lemma}
\begin{proof}
We will use a finitary version of ict-patterns, see \cite[Definition 4.21]{SiBook}, in order to calculate the dp-rank, see \cite[Proposition 4.22]{SiBook}.

Assume that $\dpr(X)\geq k$, so $\kappa_{ict}(X)>k$ and we can find formulas $\{\phi_\alpha (x_\alpha ,y)\}_{\alpha<k}$, such that for any integer $n$ there is an array $\la a_i^{\alpha}:i<n,\,\alpha<k\ra$ of tuples from $M$ with $|a_i^\alpha|=|x_\alpha|$, such that for every $\eta:k\to n$ there is a tuple $b_\eta\in X(\CM)$ such that
\[\phi_\alpha(a_i^\alpha,b_\eta)\iff \eta(\alpha)=i.\]
The same pattern gives, since $X(\CM)\subseteq Y(\CM)$, that $\kappa_{ict}(Y)>k$ so $\dpr(Y)\geq k$.
\end{proof}

We fix the uniformity on $D$, given by $$\mathscr{B}=\{U_t=\theta(D^2,t):t\in T\},$$ and consider  the $\0$-definable directed partial order defined by $t_1\leq t_2$ if $U_{t_1}\subseteq U_{t_2}$.
\begin{lemma}\label{L:generic uniformity parameters}
Let $D$ be an SW-uniformity (in some structure $\CM$), $a=(a_1,\dots, a_n)\in D^n$ and $A\sub D$. 
\begin{enumerate}
    \item For every $t_0\in T$, there exists $t\leq t_0$ such that $\dpr(a/tA)=\dpr(a/A)$.
    
    \item For every $t\in T$, there exists $a_1'\in U_{t}[a_1]\setminus \{a_1\}$  such that $\dpr(a/Ata_1')=\dpr(a/At)$.
\end{enumerate}
\end{lemma}
\begin{proof}
(1) Let $\Delta$ be the collection of all $A$-formulas $\psi(\bar x,u)$, for $\bar x=(x_1,\ldots,x_n)$ and $u$ a $T$-variable, such that for every $t\in T$, 
$\dpr(\psi(D,t))<\dpr(a/A).$

We want to realize the type
\[P(s)=\{s\leq t_0\}\cup \{\neg \psi(a,s):\psi\in \Delta\}.\]

If $P$ is consistent, and realized by $s_1\in T$ then obviously $s_1\leq t_0$.  If $\dpr(a/s_1A)<\dpr(a/A)$ then by the locality of dp-rank in $D$,  we can find some $A$-formula $\psi(\bar x,y)$, such that $\psi(\bar x,s_1)\in \tp(a/S_1A)$ and $\dpr(\psi(D,s_1))<\dpr(a/A)$.
By Fact \ref{defofrank}, there is a formula $\theta(u)$ over $A$ such that for every $t\in T$,  $\theta(t)$ holds if and only if  $\dpr(\psi(D,t))<\dpr(a/A)$. Consequently, $\psi(\bar x,u)\wedge \theta(u)\in \Delta$, contradicting the fact that $(a,s_1)$ satisfies it. Therefore a realisation of $P$ must satisfy the conclusion of the lemma.  

Assume toward a contradiction that $P$ is inconsistent. Thus there are $\phi_1,\ldots,\phi_r\in \Delta$, such that 
\[\forall t(t\leq t_0 \to \bigvee_{i} \phi_i(a,t)).\]

Let $\psi(\bar x,u):= \bigvee_i \phi_i(\bar x,u)$. Then it is still the case that $\psi\in \Delta$ and we have $\forall t(t\leq t_0\to \psi(a,t)).$
Let \[\chi(\bar x):= \exists s \forall t(t\leq s\to \psi(\bar x,t)).\]
Then $\chi$ is a formula over $A$ and as $\chi(a)$ holds then necessarily $\dpr(\chi)\geq \dpr(a/A)$.

We will reach a contradiction by showing that, in fact,  $\dpr(\chi)<\dpr(a/A)$.
Let 
\[
\widehat \psi(\bar x,u):=\psi(\bar x,u)\wedge (\forall t_1\leq u) \psi(\bar x,t_1).
\]
It is easy to verify that $\chi(\bar x)  \leftrightarrow \exists u \widehat \psi(\bar x,u).$

Let $\CM\prec\CM^*$ be an $|\CM|^+$-saturated elementary extension and let $D^*=D(\CM^*)$. By saturation  and the fact that $T$ is directed by $\le$  there is $t^*\in T(\CM^*)$ such  $t^*\leq T(\CM)$.

It follows that $\chi(D)\sub \widehat \psi(D^*,t^*)$. Indeed, if $b\in \chi(D)$ then there is $t\in T(\CM)$ such that for all $t_1\in T(\CM)$, if $t_1\leq t$ then $\psi(b,t_1)$. This remains true in $\CM^*$ hence $\widehat \psi(b,t^*)$ holds.

However, since $\psi(\bar x,u)\in \Delta$ then $\widehat \psi(\bar x,u)\in \Delta$. As a result, for every $t\in T$ $\dpr(\widehat\psi (D,t))<\dpr(a/A)$.
%As $\dpr(-)$ is definable in parameters (as the largest $k$ for which some projection into $D^k$ contains an open set), for every $t\in T(D^*)$, $\dpr(\widehat\psi(D^*,t))<\dpr(a/A)$ as well.  
In particular $\dpr(\widehat \psi(D^*,t^*))<\dpr(a/A)$ and by  Lemma \ref{L:dp-rank bounded by externally definable}, $\dpr(\chi(D))<\dpr(a/A)$, with the desired contradiction.

(2) The proof follows similar lines to (1). Let $A_1=At$ and let $\Delta$ be now all formulas $\phi(\bar x,y)$ over $A_1$, with $y$ in  the $D$-sort,  such that for all $y\in D$, $\dpr(\phi(D,y))<\dpr(a/A_1).$

As in (1) we want to realize the type 
\[P(y)=\{y\in U_{t}[a_1]\}\cup\{y\neq a_1\}\cup\{\neg\phi(a,y):\phi\in \Delta\}.\]

Assume the type is inconsistent. Then there are $\phi_1,\ldots,\phi_r\in \Delta$ such that 
\[\forall y((y\in U_{t}[a_1]\wedge y\neq a_1)\to \bigvee_i \phi_i(a,y)).\]
Let $\psi(\bar x,y)$ be the formula  $\bigvee_i \psi_i(\bar x,y) \wedge y\neq x_1$, then $\psi\in \Delta$.
For $X\sub D^k$ let $\fr(X):=\cl(X)\setminus X$ and note that if $a_1$ satisfies $\forall y((y\in U_{t}[a_1]\wedge y\neq a_1)\to \psi(a,y))$ then, as $D$ has no isolated points, $a_1\in \fr(\psi(a,D))$ 
and therefore $(a,a_1)\in \fr (\psi(D^{n+1}))$. So, as $\dpr(a_1,a/A_1)=\dpr(a/A_1)$, we see that $\dpr(a/A_1)\leq \dpr(\fr(\psi(D^{n+1})))$

On the other hand,  by \cite[Proposition 4.3]{SimWal}, \[\dpr(\fr(\psi(D^{n+1})))<\dpr(\psi(D^{n+1})).\] Thus, it must be the case that $\dpr(\psi(D^{n+1}))\geq \dpr(a/A_1)+1$.
However, $\psi\in \Delta$ hence for every $b\in \pi_{n+1}(\psi(D^{n+1}))$ (the projection to the last coordinate), we have $\dpr(\psi(D^n,b))<\dpr(a/A_1)$, contradicting subadditivity of dp-rank.
\end{proof}

We can now prove the main proposition of this section. Note that in the following $V$ is not necessarily $A$-definable.
\begin{proposition}\label{P:genos}
Let $D$ be an SW-uniformity. For every open $V\sub D^n$, $a\in V$,  and $A$ a small set of parameters, there exists $B\supseteq A$ and a $B$-definable open subset $U=U_1\times\dots\times U_n\subseteq V$ such that $a\in U$ and  $\dpr(a/B)=\dpr(a/A)$.
\end{proposition}
\begin{proof}
Assume that $a=(a_1,\ldots, a_n)$. By induction it is sufficient to prove:
\begin{itemize}
    \item[($\dagger$)] Given any open set $W\ni a_1$ there exists $B\supseteq A$ and a $B$-definable open $U\sub W$ containing $a_1$ such that $\dpr(a/B)=\dpr(a/A)$.
\end{itemize}
Indeed, given $a=(a_1,\ldots, a_n)$, and given $W\ni a$ we may assume that $W=W_1\times\dots\times W_n$. Using ($\dagger$), there exists a finite $b_1$, and an $Ab_1$-definable $U_1\sub W_1$ containing $a_1$ such that $\dpr(a/b_1A)=\dpr(a/A)$. Now replace $A$ with $A_1=b_1A$ and apply the inductive hypothesis to $W_2\times \dots \times W_n$ and $A_1$.

So we now turn to proving ($\dagger$). Fix $s_0\in T $  such that $U_{s_0}[a_1]\subseteq W$ and by the definition of a uniformity, we find $t_0\in T$, such that  $\{(x,z)\in D^2:(\exists y\in D)\left( (x,y)\in U_{t_0},\, (y,z)\in U_{t_0}\right)\}\subseteq  U_{s_0}.$ By Lemma \ref{L:generic uniformity parameters}, there exists $t_1\leq t_0$ with $\dpr(a/At_1)=\dpr(a/A)$ and $a_1'\in U_{t_1}[a_1]$ such that $\dpr(a/At_1a_1')=\dpr(a/At_1)=\dpr(a/A)$. We claim that $U_{t_1}[a_1']\subseteq U_{s_0}[a_1]\subseteq W$.

Indeed, assume that $x\in U_{t_1}[a_1']$ then $(a_1',x)\in U_{t_1}$ so  $(a_1',x)\in U_{t_0}$, as $t_1\le t_0$.
By construction $a_1'\in U_{t_1}[a_1]$, hence $(a_1,a_1')\in U_{t_1}$, so also $(a_1,a_1')\in U_{t_0}$.
It follows that $(a_1,x)\in U_{s_0}$ by our choice of $t_0$.
But then $x\in U_{s_0}[a_1]$, as we wanted.

It follows that the set $W_1=U_{t_1}[a_1']$ is a subset of $W$ containing $a_1$ (by symmetry of the uniformity) and is defined over $Ata_1'$.
\end{proof}

We shall also need the following technical corollary.

\begin{corollary}\label{C:strong GenOS}
Let $D$ be an SW-uniformity.
 For every definable $X\subseteq D^n$, $Y\subseteq X$ a definable subset, $a$ in the relative interior of $Y$ in $X$, $b\in D^k$, and $A$ a small set of parameters, there exists $B\supseteq A$ and a $B$-definable open subset $U=U_1\times\dots\times U_n\subseteq D^n$ such that $a\in U\cap X\subseteq Y$ and  $\dpr(a,b/B)=\dpr(x,y/A)$.
\end{corollary}
\begin{proof}
  Since $a$ is in the relative interior of $Y$ in $X$, there exist an open subset $V\subseteq D^n$ such that $a\in X\cap V\subseteq Y$. Consider the open set $V'=V\times D^k$. Now apply Proposition \ref{P:genos} to $V'$, $(a,b)$ and $A$.
\end{proof}

The next lemma plays an important role in our analysis of infinitesimal neighbourhoods in Section \ref{S:infi}.  Clause (1) of the lemma is well known for $n=1$. We thank Itay Kaplan for the general case. As the concepts needed for the proof are not exactly inline with the rest of the paper we postpone the proofs to the appendix.

\begin{lemma}\label{L:itaywom}
    Let $\CM$ be a structure of finite dp-rank and $\mathbb{U}\succ \CM$ a monster model.
    \begin{enumerate}
    \item Let $D$ be an SW-uniformity in $\CM$ and let $b_1,\dots,b_n$ be some tuples in $\mathbb{U}$. For every $\CM$-definable $X$, there exists $a\in X$, with $\dpr(a/M)=\dpr(X)$, such that $\dpr(ab_i/M)=\dpr(a/M)+\dpr(b_i/M)$ for all $1\leq i\leq n$.
    \item For $A\sub \mathbb U$ and $a\in \CM^n$,  there exists a small model $\CN\prec \CM$,  $A\subseteq N$, such that $\dpr(a/A)=\dpr(a/N)$.
\end{enumerate}
\end{lemma}
\begin{proof}
 See Appendix \ref{A:itay-proof}.
\end{proof}

\begin{remark}
The above lemma, too, can be viewed as a partial substitute for exchange. Indeed, it is straightforward to see that if $\dpr$ is additive in $D$, i.e. for all tuples $a,b$ from $D$ and $A$ an arbitrary set of parameters
\[\dpr(a,b/A)=\dpr(a/Ab)+\dpr(b/A),\]
then Lemma \ref{L:itaywom}(1) is true over any parameter set (not only over a model). Additivity of dp-rank is equivalent, in the context of dp-minimal structures, to exchange (see e.g., \cite[Observation 3.1, Proposition 3.2]{Simdp}). 
\end{remark}

\section{Fields locally strongly internal to various dp-minimal fields}\label{S:def fields}
The aim of this section to classify all fields $\CF$ of finite dp-rank such that some infinite definable subset of $\CF$ can be definably injected into a field $\CK$ that is either an SW-uniformity  or strongly minimal.  We show that under various assumptions such fields are definably isomorphic to  finite extensions of $K$. We work in somewhat greater generality that will be useful in the sequel.

\begin{assumption}
Throughout the \underline{entire} of Section 4 we assume that $\CM$ is $|T|^+$-saturated.
\end{assumption}

\subsection{Strong internality, and the main technical lemma}

\begin{definition}
Let $\CM$ be any (multi-sorted) structure.
\begin{enumerate}
    \item An $A$-definable set $X$ is $\emph{strongly internal to (a definable) set $Y$ over $A$}$ if there exists an  $A$-definable
    injection $f:X\to Y^n$, for some  $n\in \mathbb N$. We may omit the reference to $A$ and just say \emph{$X$ is strongly internal to $Y$.}

    $X$ is called \emph{locally strongly internal to $Y$ over $A$} if there exists some $A$-definable infinite $X'\sub X$ which is strongly internal to $Y$ over $A$. Again, we may omit the reference to $A$.

    \item Following Johnson \cite{JohnDpFin1A},  we define:  For  $X$  locally strongly internal to $Y$,  a definable set $Z\sub X$ is \emph{$Y$-critical (in $X$)} if $Z$ is strongly internal to $Y$ of maximal dp-rank, i.e., $\dpr(Z)\ge \dpr(Z')$ for all $Z'\sub X$ that is strongly internal to $Y$.
\end{enumerate}
\end{definition}

We start our investigation by proving an important technical lemma. Roughly, the lemma states that if a field of finite dp-rank is locally strongly internal to an SW-uniform structure $D$ then there exists a subset of the field strongly internal to $D$ and sufficiently closed under the field operations.

\begin{lemma}\label{L:generic linearity for K}
Let $D$ be an SW-uniformity in some (possibly multi-sorted) structure $\CM$. Let $(\CF,+,\cdot)$ be an infinite field of finite dp-rank definable field in $\CM$ and assume that $\CF$ is locally strongly internal to $D$ over $A$.
Let $Y\sub \CF$ be a $D$-critical set (over $A$) and $I\sub Y$ an $A$-definable set with $\dpr(I)=1$.
Let $(b,c,d) \in I\times Y\times Y$ be such that $\dpr(b,c,d/A)=2\dpr(Y)+1$.

Then there is $B\supseteq A$ and infinite $B$-definable sets $J\subseteq I$  and $S=Y_1\times Y_2\subseteq Y^2$, with $(b,c,d)\in J\times S$ and $\dpr(b,c,d/B)=\dpr(J\times S)=2\dpr(Y)+1$, such that for every $(x,y,z)\in J\times S$, we have $(x-b)y+z\in Y$.
\end{lemma}
	\begin{proof}
% 	We fix $A$-definable injections $g_1:I\to \CG^r$ and $g_2:Y\to \CG^l$.
	For simplicity of notation assume $A=\emptyset$. Denote  $\dpr(Y)=n$. Let $(b,c,d) \in I\times Y^2$ be as in the statement. Since dp-rank is preserved under definable bijections, we may replace $Y, I$ and, correspondingly, $(b,c,d)$ by their images under any $\0$-definable bijections. As we are considering $I\times Y$ with all of its induced structure, and  $I, Y$ are strongly internal to $D$,  witnessed by $\0$-definable functions, we may identify $I$ and $Y$ (and therefore also $I\times Y$) and $(b,c,d)$ with definable sets and tuples in $D$.
	%By replacing $Y$ and $I$ with respective relatively open subsets and applying \cite[Proposition 4.6]{SimWal} we may, in fact identify $I$ with an infinite open subset of $\CG$ and $Y$ with an open subset of $\CG^n$.
	
		Note that
	\[
	    2n+1=\dpr(b,c,d)\leq \dpr(b,d/c)+\dpr(c)\leq \dpr(b,d)+\dpr(c)\leq 2n+1
	\]
	so that $\dpr(b,d/c)=n+1$, $\dpr(c)=n$. Similarly, $\dpr(b,c)=n+1$ and $\dpr(b/c)=1$.
	
	For $(x,y, z) \in I\times Y\times Y$ consider the function $f_y(x,z)=xy-z$. Let $e=f_c(b,d)$.

\begin{claim}
 $b\notin \acl(c,e)$.
\end{claim}
\begin{claimproof}
Assume towards a contradiction that $b\in\acl(c,e)$. By Proposition \ref{P:genos}, we can find some $B$ such that $b\in\dcl(Bc,e)$ and $\dpr(b,c,d/B)=\dpr(b,c,d)$. Indeed, we can first find $U\ni b$ such that $b$ is the only realization of $\tp(b/c,e)$ in $U$ and then apply Proposition \ref{P:genos} (for the tuple $(b,c,d)$ and any open box whose projection onto the $b$-coordinate is $U$).

Let $\varphi(x,e,c)$ be the algebraic formula isolating $\tp(b/Bc,e)$, in particular, $\varphi(x,e,c)$ implies that $x\in I$. By the definition of $f_c$, $d\in \dcl(b,c,e)$, therefore  $\phi(x,e,c)\wedge f_c(x,y)=e$ is an algebraic formula isolating $\tp(b,d/Bc,e)$. In fact, $(b,d)$ is the only pair of elements realizing this type. Hence, $\exists!(x,y)(\phi(x,e,c)\wedge f_c(x,y)=e)$. By compactness, there is a formula $\psi(z,c)\in \tp(e/Bc)$ implying $\exists!(x,y)(\phi(x,z,c)\wedge f_c(x,y)=z)$. In other words, for $X:=\psi(\CF,c)$, there is a $Bc$-definable injective function $F$ from $X$ into $I\times Y$, sending $e$ to $(b,d)$.

The image of $F$ in $I\times Y$ is a $Bc$-definable set containing $(b,d)$ and since $\dpr(b,d/Bc)=n+1$, we have
 $\dpr(\mathrm{Im}(F))=n+1$ and consequently $\dpr(\dom(F))>n$.
As $\dpr(\mathrm{dom}(F))>\dpr(Y)$ and $\dom(F)\sub \CF$ this contradicts $Y$ being $D$-critical.
%Recall that $(g_1,g_2):I\times Y\to \CG^r\times \CG^l$ is a $\0$-definable injection, so $H=(g_1,g_2)\circ G:\mathrm{Dom}(G)\to \CG^{r+l}$  is a one-to-$N$ correspondence,
%If $\CG$ expands a linear order then by choosing some elements $c_1,\dots, c_N$, $H$ induces a map $\dom(G)\times \{c_1,\dots,c_N\}\to \CG^{r+l}$ mapping $\la t,c_i\ra $ to the $i$-th element (in the lexicographic order) of $H(t)$. We arrive to a similar contradiction.
\end{claimproof}

We  conclude that $b\notin \acl(c,f_c(b,d))=\acl(c,e)$ and in particular the projection of $f_c^{-1}(e)\subseteq I\times Y$ on the first coordinate,  call it $I'$,  is infinite and contains $b$.
%Notice that by the definition of $f_c$, for every $b'\in I$, and $d_1\neq d_2\in Y$, we have $f_c(b_1,e_1)\neq f_c(b_1,e_2)$,
%thus the projection of $f_c^{-1}(e)$ on the first coordinate, call it $J$, is infinite.
So $\dpr(b/ce)=\dpr(I')=1$, and by definition for every $x\in I'$ there is $z\in Y$ with $xc-z=bc-d=e$.  Since $\CF$ is a field, the map from $I'$ to $Y$, taking $x\in I'$ to $z=xc-e=(x-b)c+d$, is injective.

By Fact \ref{rel interior}, $b$ is in the relative interior of $I'$ in $I$. We may now apply Corollary \ref{C:strong GenOS} to the definable sets $I'\subseteq I$ and the elements $b$ and $(c,d)$, to get a set of parameters $A_1$ and an $A_1$-definable subset $I''\subseteq I'$ containing $b$, such that $\dpr(b,c,d/A_1)=1+2n$.

 Let
\[
S=\{(y,z)\in Y^2: (\forall x\in I'')((x-b)y+z\in Y)\}.
\]
Note that $(c,d)\in S$ and that $S$ is $bA_1$-definable. Because $(c,d)\in S$ and $\dpr(c,d/bA_1)=2n$ clearly $\dpr(S)=\dpr(Y^2)=2n$. Also, by Fact \ref{rel interior} again, $(c,d)$ is in the relative interior of $S$ in $Y^2$.

Now consider $I''\times S\subseteq I\times Y^2$. By the above, $\dpr(I''\times S)=\dpr(I\times Y^2)$, and $(b,c,d)$ is in the relative interior of $I''\times S$ in $I\times Y^2$ so Corollary \ref{C:strong GenOS} applies (with $x=(b,c,d)$). We can thus find a set $B\supseteq A_1$ and $B$-definable subsets $J\subseteq I''$, $Y_1\times Y_2\subseteq S$, with $(b,c,d)\in J\times Y_1\times Y_2$ such that $\dpr(b,c,d/B)=\dpr(b,c,d)=2n+1$.
%Because $\la c,d,b\ra$ and  $\la b,c,d\ra $ are inter-definable over $\0$
%it follows that $\dpr(c,d/B)=2n$, so the proof is completed.
\end{proof}

The next corollary will help us late on to show that our  field $\CF$ is not locally strongly internal to various sorts. It is an analogue of the fact that a linear o-minimal structure cannot support a definable field structure.

\begin{corollary}\label{C:not-internal-to locally affine}
Let $(G,\oplus)$ be an SW-uniformity (in $\CM$) supporting a definable abelian group structure.  Assume that $G$ satisfies the following: For every definable (partial) $f:G^m\to G$, $m\in \mathbb N$,    whose domain is open there exists an open definable $R\subseteq \dom(f)$, a definable group homomorphism  $L:G^n\to G$ and $e\in G$ such that for every $y\in R$, $f(y)=L(y)\oplus e$. Then no definable infinite field is locally strongly internal to $G$.
\end{corollary}
\begin{proof}
Assume  towards contradiction that an infinite definable field $\CF$ is locally strongly internal to $G$ and let $Y\sub \CF$ be $\CG$-critical in $\CF$ with $\dpr(Y)=n$.
 By Remark \ref{R:product of one-dim in SW}, there is a definable dp-minimal $I\sub \CF$. Assume that $Y, I$ and the injection of $Y$  into $G^k$ for some $k$,  are all defined over $A$.

By Lemma \ref{L:generic linearity for K} (applied after fixing an appropriate 
tuple  $(b,c,d)\in I\times Y^2$), there are  definable  $J\subseteq I$ and $S\subseteq Y^2$ such that $\dpr(J\times S)=1+2n$ and $f(x,y,z)=(x-b)y+z$ maps $J\times S$ into $Y$. By strong internality combined with Fact \ref{localhom} we can, after possibly replacing $I$ and $Y$ with  subsets of the same dp-rank (defined, possibly, over additional parameters)  identify $I$ with a subset of $G$ and $Y$ with a subset of $G^n$. Thus,  we may assume that $f:G^{1+2n}\to G^n$.

By our assumption, we may inductively find an additive (with respect to $\oplus$) definable function $L:G^{1+2n}\to G^n$ and $e\in G^n$ such that $f(x,y,z)$ and $L(x,y,z)\oplus e$ agree on some $X\subseteq J\times S$ with $\dpr(X)=2n+1$. Fix some $J_0\times S_0\subseteq X$ with $\dpr(J_0)=1$ and $\dpr(S_0)=2n$.

The map $(y,z)\mapsto f(0,y,z)$ maps $S_0$ into $Y$, where here $0=0_G$. Since $\dpr(S_0)=2n>n=\dpr(Y)$ the map cannot be injective and hence there are $(y_1,z_1)\neq (y_2,z_2)$ for which $f(0,y_1,z_1)=f(0,y_2,z_2)$.

It follows that $L(0,y_1,z_1)=L(0,y_2,z_2)$. But then for every $x\in G$, \[L(x,y_1,z_1)=L(0,y_1,z_1)\oplus L(x,0,0)=L(x,y_2,z_2)\] Thus also $f(x,y_1,z_1)= f(x,y_2,z_2)$ for any $x\in J_0$.

On the other hand, by the definition of $f$ in the field $\CF$, for every $(y_1,z_1) \neq (y_2,z_2)$ there is at most one $x$ such that $f(x,y_1,z_1)=f(x,y_2,z_2)$. Contradicting the fact that $J_0$ is infinite.
\end{proof}
%\todo{It may be of interest to mention that unlike the o-minimal case, infinite %fields can be interpretable in an SW-uniformity as in the the corollary, e.g., %in K/O if \bk is infinite.} 
%%\todo{Kobi:I did not understand the above remark. please suggest a sentence}
%\todo{Kobi: I added something in $K/\CO$}

\subsection{The subgroup of infinitesimals}\label{S:infi}

\begin{assumption}
In addition to our $|T|^+$-saturation assumption, throughout this section $\CM$ is any first order (multi-sorted) structure and  $D$ an $\CM$-definable SW-uniformity. For ease of presentation, assume that $\CM$ has one distinguished sort whose universe is $M$.
\end{assumption}

\begin{definition}
For any $Z\subseteq M^n$ and definable injective  $g:Z\to D^m$, let
$\tau_{Z,g}$ be the topology on $Z$ given by $\{g^{-1}(U): U\subseteq D^m \text{ is $M$-definable open}\}$
\end{definition}

We observe that because $D^n$ has a definable basis for its topology so does $\tau_{Z,g}$ (for any $Z$ and $g$). Moreover, if $Z_1,Z_2\subseteq M^n$ and $g_i:Z_i\to D^{m_i}$ are definable injections (for $i=1,2$) then  the topology $\tau_{Z_1\times Z_2,g_1\times g_2}$ is the product topology $\tau_{Z_1,g_1}\times \tau_{Z_2,g_2}$.

It follows immediately from Fact \ref{F:cont} and the above observation that:
\begin{lemma}\label{F:cont-points}
    If $g_i: Z_i\to D^{m_i}$ are definable injections ($i=1,2$) and  $f: Z_1^k\to Z_2^l$ is a definable (partial)  function then the set of continuity points $C_f$ of $f$ with respect to $\tau_{Z_i,g_i}$ is definable and  $\dpr(\dom(f)\setminus C_f)<\dpr(\dom(f))$.
\end{lemma}

We thus have:
\begin{lemma}\label{unique topology}
Let  $Z\sub M^n$ be definable and  $g:Z\to D^m$, $h:Z\to D^k$ two $A$-definable injections. Then,
$\tau_{Z,g}$ and $\tau_{Z,h}$ agree at every $z\in Z$ with $\dpr(z/A)=\dpr(Z)$.
Namely, there is a common basis for the $\tau_{Z,g}$-neighbourhoods and the $\tau_{Z,h}$-neighbourhoods for such a $z\in Z$.
\end{lemma}
\begin{proof}
Apply Lemma \ref{F:cont-points} to $\id:Z\to Z$.
\end{proof}

\begin{definition}
For $Z\sub M^n$ definable, $g:Z\to D^m$ a definable injection, and $d\in Z$,
 let $\nu_{Z,g}(d)$ be the partial global type given by all definable $\tau_{Z,g}$-open sets containing $d$.
We call it the {\em infinitesimal neighborhood of $d$ with respect to $\tau_{Z,g}$}.
\end{definition}

It follows from the above discussion that:
\begin{remark}\label{R:product of nu}
If $g_i:Z_i\to D^{m_i}$ (for $i=1,2$) are $A$-definable injections  $(d_1,d_2) \in Z_1\times Z_2$ then $\nu_{Z_1,g_1}(d_1)\times \nu_{Z_2,g_2}(d_2)=\nu_{Z_1\times Z_2,g_1\times g_2}(d_1,d_2)$.
% (see above discussion on $\tau_{Z_1,g_1}\times \tau_{Z_2,g_2}$).
\end{remark}

By Lemma \ref{unique topology}, we have:
\begin{corollary}\label{C:not depend on injection}
If $Z$ is strongly internal to $D$ over $A$ and $d\in Z$ is such that $\dpr(d/A)=\dpr(Z)$ then $\nu_{Z,g}(d)$ does not depend on the choice of the definable injection $g$ (over $A$).
 \end{corollary}

\begin{notation}
    If $Z$ is strongly internal to $D$ over $A$ and $d\in Z$ with $\dpr(d/A)=\dpr(Z)$, we let $\nu_Z(d):=\nu_{Z,g}(d)$ for some (equivalently, any) $A$-definable injection $g:Z\to D^m$ (some $m$). By Corollary \ref{C:not depend on injection} this is well defined.
\end{notation}

We observe also that $\nu_Z(d)$ does not depend on $Z$ (but only on its germ at $d$) in the following sense:
\begin{lemma}\label{internal-locally the same}
Assume that $Z\sub M^n$ is strongly internal to $D$ over $A$, witnessed by $g$, and $Z_1\sub Z$ is $A$-definable with $\dpr(Z_1)=\dpr(Z)$.
If $d\in Z_1$ is such that $\dpr(d/A)=\dpr(Z)$ then $\nu_{Z_1}(d)=\nu_{Z}(d)$.
\end{lemma}
\begin{proof}
By Lemma \ref{L:locally the same}, the topologies $\tau_{Z,g}$ and $\tau_{Z_1,g}$ agree on a neighborhood of $d$, thus $\nu_Z(d)=\nu_{Z_1}(d)$.
\end{proof}

Finally, we note;
\begin{lemma} \label{function}
%\begin{enuemrate}
%\item
Let $Y_1,Y_2$ be definable sets  strongly internal to $D$ over $A$. If $f:Y_1\to Y_2$ is an $A$-definable partial function, $\dpr(\dom(f))=\dpr(Y_1)$,  and $a\in \dom(f)$ with $\dpr(a/A)=\dpr(\dom(f))$, then $\nu_1(a)\vdash f^{-1}(\nu_2(f(a))):=\{f^{-1}(U): U\in \nu_2(f(a))\}$. I.e., if $b\models \nu_1(a)$ then $f(b)\models \nu_2(f(a))$

    \end{lemma}
    \begin{proof}
    By Lemma \ref{F:cont-points}, $f$ is continuous at $a$ with respect to $\tau_{Y_1,g}$ and $\tau_{Y_2,h}$ (for any $A$-definable $g,h$ witnessing the strong internality of $Y_1$, $Y_2$, respectively). The conclusion follows.
    \end{proof}

We now prove a general statement.
\begin{lemma} \label{addition closure}
Let $(H,\cdot)$ be a definable group in  $\CM$ and
assume that $Y_1,Y_2,Y_3\sub H$ are  $\CM$-definable sets, strongly internal to $D$, all definable over some model $\CN\prec \CM$, with $\dpr(Y_1)=\dpr(Y_2)=\dpr(Y_3)=k$, and assume that $Y_1\cdot Y_2\sub Y_3$.
Then
% (the operations taken in the group $H$)
\begin{enumerate}
\item  For every $c\in Y_1$ and $d\in Y_2$ such that $\dpr(c,d/N)=2k$, we have $c^{-1}\cdot \nu_{Y_1}(c)=\nu_{Y_2}(d)\cdot d^{-1}$, and for every $d_1\in Y_2$ with $\dpr(d_1/A)=k$ we have  $\nu_{Y_2}(d)\cdot d^{-1}=\nu_{Y_2}(d_1)\cdot d_1^{-1}$.
    \item  For every $d\in Y_2$ such that $\dpr(d/N)=k$, the partial type $\nu_{Y_2}(d)\cdot d^{-1} $ is a type definable subgroup of $H$.
    \end{enumerate}
\end{lemma}
\begin{proof}
It is convenient to work in an $|\CM|^+$-saturated elementary extension $\widehat \CM$ of $\CM$ and with the realizations in $\widehat \CM$ of the infinitesimal types. We start with some general observations.

Consider the function $F:Y_1\times Y_2\to Y_3$, $F(y_1,y_2)=y_1\cdot y_2$.
Applying Lemma \ref{function} to $F$ (and using Remark \ref{R:product of nu}), we see that $F(\nu_{Y_1}(c)\times \nu_{Y_2}(d))\sub \nu_{Y_3}(c\cdot d)$. Because the dp-rank is preserved under the definable bijection $(x,y)\mapsto (x,xy)$, then $\dpr(c,c\cdot d)=\dpr(c,d)=2k=\dpr(Y_1\times Y_3)$.

 Consider also the function $R:Y_1\times Y_3\to H$, $R(x,z)=x^{-1}\cdot z$. Since $R(c,c\cdot d)=d\in Y_2$, it follows that the set
$W=\{(x,z) \in Y_1\times Y_3:R(x,z)\in Y_2\}$ contains $(c,c\cdot d)$ and has $\dpr=2k$. Thus, by Lemma \ref{internal-locally the same}, $\nu_{W}(c,c\cdot d)=\nu_{Y_1\times Y_3}(c,c\cdot d)$ and by Lemma \ref{function} (and Remark \ref{R:product of nu}), $R$ sends (realisations of) $\nu_{Y_1}(c)\times \nu_{Y_3}(c\cdot d)$ to (realisations of) $\nu_{Y_2}(d)$.
It follows that for every $c_1\in \nu_{Y_1}(c)$, the function $y\mapsto c_1\cdot y$ is a bijection between $\nu_{Y_2}(d)$ and $\nu_{Y_3}(c\cdot d)$. Indeed, $F(c_1,-)$ is a function from $\nu_{Y_2}(d)$ into $\nu_{Y_3}(c\cdot d)$, whose inverse is $R(c_1,-)$.

Since $F$ maps realisations of $\nu_{Y_1}(c)\times \nu_{Y_2}(d)$ onto realisations of $\nu_{Y_3}(c\cdot d)$ it follows that for every $d'\in \nu_{Y_2}(d)$, the function $x\mapsto x\cdot d'$ is a bijection between realisations of $\nu_{Y_1}(c)$ and realisations of $\nu_{Y_3}(c\cdot d)$.
In particular,  $c\cdot \nu_{Y_2}(d)=\nu_{Y_1}(c)\cdot d=\nu_{Y_3}(c\cdot d)$ and we can therefore conclude that $\nu_{Y_2}(d)\cdot d^{-1}=c^{-1}\cdot \nu_{Y_1}(c)$, proving the first clause of (1).

Assume next that $d_0,d_1\in Y_2$ and $\dpr(d_0/N)=\dpr(d_1/N)=k$. By Lemma \ref{L:itaywom}(1), we can find $c\in Y_1$ such that $\dpr(c,d_0/N)=\dpr(c,d_1/N)=2k$.
And then, by what we just saw,  $\nu_{Y_2}(d_0)\cdot d_0^{-1}=c^{-1}\cdot \nu_{Y_1}(c)=\nu_{Y_2}(d_1)\cdot d_1^{-1}$.

To prove (2)  we need to show that  $\nu_{Y_2}(d)\cdot d^{-1}$ is a subgroup of $H$, for every $d\in Y_2$ with $\dpr(d/N)=k$. Given $a,b\in \nu_{Y_2}(d)$, we need to show that $(a\cdot d^{-1})\cdot (b\cdot d^{-1})^{-1} =a\cdot b^{-1}$ is also in $\nu_{Y_2}(d)\cdot d^{-1}$.

By our above observations, the maps $y\mapsto c\cdot y$  is a bijection of $\nu_{Y_2}(d)$ and $\nu_{Y_3}(c\cdot d)$ and the map $x\mapsto x\cdot b$ is a bijection of $\nu_{Y_1}(c)$ and $\nu_{Y_3}(c\cdot d)$.   Hence, there is $c_1\in \nu_{Y_1}(c)$ such that $c\cdot a=c_1\cdot b$.

It follows that $a\cdot b^{-1}=c^{-1}\cdot c_1\in c^{-1}\cdot \nu_{Y_1}(c)$.
By what we just showed, $c^{-1}\cdot \nu_{Y_1}(c)=\nu_{Y_2}(d)\cdot d^{-1}$, hence $a\cdot b^{-1}\in \nu_{Y_2}(d)\cdot d^{-1}$.
\end{proof}

We now apply the above to definable fields.

\begin{proposition}\label{P:infinit}
Let $(\CF,+,\cdot)$ be a definable field in $\CM$.
%Let $I,Y\subseteq \CF$ be $A$-strongly internal to $\CF$,  with $I$ dp-minimal and $Y$ $\CG$-critical.
\begin{enumerate}
    \item Assume that $\CF$ and $D$ are definable over a small model $\CN_0\prec \CM$ and let $Y\sub \CF$ be $D$-critical and strongly internal to $D$ over $\CN_0$. Then for every $d\in Y$ such that $\dpr(d/N_0)=\dpr(Y)$, the partial type $\nu_{Y}(d)-d$ is a  subgroup of $(\CF,+)$. Moreover, the subgroup is independent of the choice of $d$, and we denote it $\nu_{Y}$.

    \item For every $D$-critical $Y_1,Y_2\subseteq \CF$,  $\nu_{Y_1}=\nu_{Y_2}$.

    \item Let $\nu:=\nu_Y$ be the partial type associated to some (any) $D$-critical $Y$, as implied by (2). Then  $\nu$ is invariant under multiplication by scalars from $\CF$.

    % \item Let $\widehat \CM$ be an $|M|^+$-saturated elementary extension of $\CM$. Then $\nu(\widehat M)$ is closed under addition and multiplication.
    %As a result, we may write $\nu$ for $\nu_Y$.
\end{enumerate}
\end{proposition}
\begin{proof}
(1) Since $Y$ is strongly internal to $D$ over $\CN_0$, and $D$ is dp-minimal there exists some $\CN_0$-definable $I\sub Y$ that is dp-minimal (Remark \ref{R:product of one-dim in SW}).
We assume that $Y$ is a subset of $D^k$ for some $k$ and let $n=\dpr(Y)$. By Corollary \ref{unique topology}, for every  $d\in Y$ with $\dpr(d/N_0)=n$, the infinitesimal neighbourhood $\nu_Y(d)$ (and therefore also $\nu_Y(d)-d$) does not depend on the choice of the embedding of $Y$ in $D$ over $N_0$.

We  prove simultaneously that $\nu_Y(d)-d$ is a group and that it does not depend on the choice of $d$. Our intention is to apply Lemma \ref{addition closure} to the definable group $(\CF,+)$. So let $d,d'\in Y$  be such that $\dpr(d/N_0)=\dpr(d'/N_0)=n$. Our first observation is that we may assume that $\dpr(d,d'/N)=2n$. Indeed, by Lemma \ref{L:itaywom}(1) we can find $c\in Y$ such that $\dpr(d,c/N)=\dpr(d',c/N)=2n$. So if we prove the  equality of groups for $d,c $ and $d',c$ then we would get that $\nu_Y(d)-d=\nu_Y(c)-c=\nu_Y(d')-d'$. So from now on we assume $\dpr(d,d'/N)=2n$.

By Lemma \ref{L:itaywom}(1), there are $(b,c)\in I\times Y$ such that $\dpr(b,c,d,d'/N_0)=1+3n$. In particular
\[
\dpr(b,c,d/N_0)=\dpr(b,c,d'/N_0)=1+2n.
\]
We first prove:
\begin{claim}
There are relatively open sets $\bar J\sub I$, containing $b$, $\bar Y_1\sub Y$ containing $c$, and $Y_2,Y_2'\sub Y$ containing $d$ and $d'$, respectively, such that for every $(x,y,z)\in \bar J\times \bar Y_1\times (Y_2\cup Y_2')$, we have
\[(x-b)\cdot y+z\in Y.\]
\end{claim}
% By Gen-ex, there exists $(b,c)\in I\times Y$  such that $\dpr(b,c,d/A)=\dpr(b,c,d'/A)=2n+1$.
%We first handle $(b,c,d)$:
\begin{claimproof}
Applying Lemma \ref{L:generic linearity for K} to  $(b,c,d)$ we obtain $B\supseteq N_0$ and $B$-definable $J\times Y_1\times Y_2\subseteq I\times Y^2$, with $(b,c,d)\in J\times Y_1\times Y_2$ such that $\dpr(b,c,d/B)=1+2n$ and such that the map $(x,y,z)\mapsto (x-b)y+z$ sends $J\times Y_1\times Y_2$ into $Y$.
% Next we apply Lemma \ref{L:generic linearity for K} to  $(b,c,d')$
Similarly we obtain $B'\supseteq N_0$ and $B'$-definable $J'\times Y'_1\times Y'_2\subseteq I\times Y^2$, with $(b,c,d')\in J'\times Y'_1\times Y'_2$ such that $\dpr(b,c,d'/B')=1+2n$ and such that $(x,y,z)\mapsto (x-b)y+z$  sends $J'\times Y'_1\times Y'_2$ into $Y$.

Since $b\in J$ and $\dpr(b/B)=1$, $b$ is in the relative interior of $J$ in $I$, by Fact \ref{rel interior}. Likewise, $b\in J'$ is in the relative interior of $J'$ in $I$. Hence $b$ is in the relative interior of $\bar J=J\cap J'$ in $I$. By similar arguments, $c\in Y_1\cap Y_1'$ is in the relative interior of $\bar Y_1=Y_1\cap Y_1'$ in $Y$ and $d\in Y_2$ (respectively $d'\in Y_2'$) is in the relative interior of $Y_2$ (respectively $Y_2'$) in $Y$. Consequently, $(b,c,d,d')$ is in the relative interior of $\bar J\times \bar Y \times Y_2\times Y_2'$ in $I\times Y^3$, and by replacing the sets with their relative interior we may assume that $\bar J$ is realtively open in $J$ and  $\bar Y_1$,  $Y_2$, $Y_2'$ are realtively open in  $Y$. This ends the proof of the claim.
\end{claimproof}

By Corollary \ref{C:strong GenOS} and Lemma \ref{L:itaywom}(2) there is some small model $\CL\supseteq \CN_0$ and $\CL$-definable subsets $\widehat J\subseteq \bar J$, $\widehat Y_1\subseteq \bar Y_1$, $W_2\subseteq Y_2$ and $W_2'\subseteq Y_2'$ such that \[(b,c,d,d')\in \widehat J\times \widehat Y_1\times W_2\times W_2',\] and $
\dpr(b,c,d,d'/L)=\dpr(b,c,d,d'/N)=3n+1$. Thus for all $(x,y)\in \widehat J\times \widehat Y_1$ and for every $z\in W_2\cup W_2'$, \[(x-b)y+z\in Y.\]

By the sub-additivity of dp-rank, we must have $\dpr(c,d,d'/Lb)=3n$, so by Lemma \ref{L:itaywom}(2), we can find a small model $\CK\supseteq Lb$ such that $\dpr(c,d,d'/K)=\dpr(c,d,d'/Lb)=3n$.

Let $b_1\in \widehat J(\CK)$, with $b_1\neq b$, let $W_1=(b_1-b)\widehat Y_1$ and let $c'=(b-b_1)c\in W_1$. By our assumptions, $W_1+W_2\sub Y$ and $W_1+W_2'\sub Y$.
Note that $W_1,W_2,W_2'$ are $\CK$-definable and since $b_1-b$ is invertible (i.e. $x\mapsto (b_1-b)x$ is invertible) we have $\dpr(W_1)=\dpr(\widehat Y_1)=n$ and also $\dpr(W_2)=\dpr(W_2')=\dpr(Y)=n$.

Finally, $\dpr(d/K)=\dpr(d'/K)=n$, $\dpr(c,d/K)=\dpr(c,d'/K)=2n$ and since $c'=(b-b_1)c$ and $c$ are interdefinable over $K$ (which contains $b,b_1$), it follows that $\dpr(c',d/K)=\dpr(c',d'/K)=2n$.

We can now apply Lemma \ref{addition closure} to both $(c',d)$ and $(c',d')$,
in $W_1\times W_2$ and $W_1\times W_2'$, respectively. The assumptions of the lemma are satisfied (with $W_1,W_2$ in the respective roles of  $Y_1,Y_2$,  and the parameter set $K$ replacing $N$). Thus, $\nu_{W_2}(d)-d$ is a subgroup of $(\CF,+)$ which equals  $\nu_{W_1}(c')-c'$. Repeating the same argument with $W_1, W_2'$ we get that  $\nu_{W_2}(d')-d'$ is a subgroup of $(\CF,+)$ which equals  $\nu_{W_1}(c')-c'$. It follows that \[\nu_{W_2}(d)-d=\nu_{W_2'}(d')-d'.\]
 Finally, since $W_2,W_2'\subseteq Y$, $\dpr(W_2)=\dpr(W_2')=\dpr(Y)=n$ and $\dpr(d/K)=\dpr(d'/K)=n$, it follows from Lemma \ref{internal-locally the same}  that $\nu_{W_2}(d)=\nu_Y(d)$ and $\nu_{W_2'}(d')=\nu_{Y}(d)$, and therefore \[\nu_Y(d)-d=\nu_Y(d')-d'.\]

 We conclude that for every $d\in Y$ such that $\dpr(d/N_0)=n$, the set $\nu_Y(d)-d$ is a subgroup of $(\CF,+)$ and that this group is independent of the choice of $d$. We denote this subgroup by $\nu_Y$.

(2) Let $Y_1,Y_2\subseteq \CF$ be $D$-critical sets over small models $\CN_0$ and $\CM_0$, respectively. Let $Y'=Y_1\cup Y_2$. Note that $Y'$ is also $D$-critical and strongly internal to $D$ over $\CL$, some small model containing $N_0M_0$.
For every $d\in Y_1$ with $\dpr(d/L)=n$, it follows from  Lemma \ref{internal-locally the same} that
$\nu_{Y_1}(d)=\nu_{Y_1\cup Y_2}(d)$. Starting with $d'\in Y_2$ with $\dpr(d'/L)=n$ we similarly conclude that $\nu_{Y_2}(d)=\nu_{Y_1\cup Y_2}(d)$. What we have proved in (1) implies that
 $\nu_{Y_1}=\nu_{Y_1\cup Y_2}=\nu_{Y_2}$.

(3) Let $c\in \CF$ and let $Y\subseteq \CF$ be a $D$-critical set over some small model $\CN_0$.
    Let $\CK\supseteq N_0c$ be a small model and $d\in Y$ be such that $\dpr(d/K)=n$. So also $\dpr(c\cdot d/K)=n$. The function $x\mapsto cx$ sends $Y$ to $cY$, and by Lemma \ref{function}, it sends $\nu_Y(d)$ onto $\nu_{cY}(cd)$. Hence,
    \[c(\nu_Y(d)-d)=c\nu_Y(d)-cd=\nu_{cY}(cd)-cd=\nu.\]
    \end{proof}
    
\begin{remark}
 One can also show that if $\widehat \CM$ is an $|M|^+$-saturated elementary extension of $\CM$ then $\nu(\widehat M)$ is closed under addition and multiplication. As this result will not be needed, we omit the proof.
\end{remark}

\subsection{Differentiability in SW-uniform fields}
Our goal in this section is, under suitable assumptions, to put a differential structure on the group of infinitesimal $\nu$ constructed in the previous subsection. Before doing that, we have to discuss briefly the notion of differntiability in the context of SW-uniform fields, which are either real closed or valued.

For a real closed field  $L$ and $a=(a_1,\dots,a_n)\in L^n$, we let $|a|=\max_i \{|a_i|\}$. If $(L,v)$ is a valued field we let $v(a)=\min_i \{v(a_i)\}$.

\begin{definition}
Let $L$ be either a real closed field  or a valued field with valuation $v$.  Given $U\subseteq L^n$ open, a map $f:U\to L^m$ is \emph{differentiable} at $x_0\in U$ if there exists a linear map $D_{x_0}f:L^n\to L^m$ such that:

\noindent In the real closed case:
\[\lim_{x\to x_0}\frac{|f(x)-f(x_0)-(D_{x_0}f)\cdot (x-x_0)|}{|x-x_0|}=0,\]
and in the valued case:
\[\lim_{x\to x_0}\left[ v(f(x)-f(x_0)-(D_{x_0}f)\cdot (x-x_0))-v(x-x_0)\right]=\infty.\]

If $f$ is differentiable at a point $x_0$ we  say that $x_0$ is a $\CD^1$-point of $f$ and that $f$ is $\CD^1$ at $x_0$.
 \end{definition}

Exactly as in real analysis,  if  $f$ is differentiable at $x_0$ then $D_{x_0}f$ is represented by the Jacobian matrix of $f$ at $x_0$. It follows that, in expansions of both valued fields and real closed fields, if $f$ is definable, the set of points in its domain where $f$ is  $\CD^1$ is definable as well.

The next remark is not needed in the sequel, but simplifies the discussion, as it shows that the notion of differentiability does not depend on the valuation, but only on the topology it generates. Thus, if $K$ is a convexly valued real closed field this notion of differentiability coincides with the standard definition: 

\begin{remark} Notice that if $f:K\to K$ is a one-variable function then the above definition of $d:=D_{x_0}(f)$ is purely topological, since it requires the limit of the quotient $\frac{f(x_0)-f(x_0)-d}{x-x_0}$ to be $0$ in $K$.
Thus, if $v,v'$ valuations on $K$ (possibly one of them the absolute value, if $K$ is real closed)  which generate the same topology, then for one variable functions, $D_{x_0}(f)$ is the same when computed with respect to $v,v'$.
It follows, for $f:K^n\to K^m$, that the Jacobian Matrix of $f$ at $x_0$ does not depend on $v$, and as a result $D_{x_0}(f)$ does not depend on $v$ either.

%Let $(D_{x_0}f,v)$ denote the differential of $f$ at $x_0$ with respect to $v$, if it exists. Then $(D_{x_0}f,v)=(D_{x_0}f,v')$, and in particular one exists if and only so does the  other. Indeed, this is trivial if $f$ is a single  variable function. Thus  for any function $f$, the Jacobian Matrix of $f$ at $x_0$ does not depend on $v$, and the conclusion follows readily. 
\end{remark}

 In the classical setting of $\Rr$ or $\Cc$ (and also in  o-minimal expansions of fields) the existence of continuous partial derivatives (denoted $\CC^1$) is a sufficient condition for differentiability. We do not know whether  this is a sufficient condition in the more general setting we are interested in.

To address this we need an additional axiom (which is elementary in any expansion of a real closed or a valued field):
\begin{gen-dif}
For every definable open $D\subseteq K^n$, and definable $f:D\to K$,  the set of points $x\in D$ such that $f$ is not $\CD^1$ at $x$ has empty interior.
\end{gen-dif}

\begin{assumption}
From now on, whenever we say a definable field satisfies {\gendif} we implicitly mean that it is either real closed or supports a (fixed) definable valuation.
\end{assumption}

\begin{remark}
 As we noted in Example \ref{E:SW-uniformities}(4), it follows from Johnson's work that a dp-minimal field supports a unique field topology (with definable basis), which gives rise to an SW-uniformity. Thus, by Example \ref{E:SW-uniformities}(3), any two definable valuations give rise to two will generate the same topology, and if $K$ is a real closed valued field, then the valuation and the order will generate the same topology. 
\end{remark}

%\begin{remark}
%It follows from Johnson's Theorem (Fact \ref{F:Johnson DP}), that every SW-uniform field is either real closed or supports a definable (henselian) valuation. By \cite[Proposition 4.25]{JaSiWa2015},  every definable valuation on a dp-minimal field is henselian, implying \todo{need a reference} that all definable valuations on a dp-minimal field generate the same topology. If the field is real closed, such a valuation must be convex, and therefore induces the same topology as the order (so the same topology as the absolute value).
%\end{remark}

We are now ready to prove that under {\gendif} the group $\nu$ of infinitesimals can be endowed with a structure of a ``Lie-group'' with respect to $K$, namely, that $\nu$ has a structure of a differential manifold respecting the group operations. Conveniently, the manifold structure can be given by a single chart. 
The proof follows a strategy due to Ma\v{r}\'{\i}kov\'{a} \cite{MarikovaGps}, generalising Pillay's original work  \cite{Pi5}, where she proved that invariant groups in o-minimal structures admit a unique definable group topology.  In our proof we replace continuity of functions in her work with differentiability.
%We use a very similar proof that would work for any definable property that is generically true of all definable functions (in the %present case replacing continuity in Marikova's work with differentiability). in a multi-sorted structure $\CM$ 
\begin{proposition}\label{P:C1 structure}
Let $\CM$ be a $|T|^+$-saturated multi-sorted structure, $(K,+,\cdot)$ be a definable SW-uniform field satisfying {\gendif} and let $\CF$ be a infinite definable field locally strongly internal to $K$.

    Let $Y\subseteq \CF$ be $K$-critical, $g:Y\to K^n$ an injective $A$-definable map with open image, and let  $\widehat \CM\succ \CM$ be  an $|\CM|^+$-saturated elementary extension. Then there exists $c\in Y$, with $\dpr(c/A)=\dpr(Y)$, such that,
    \begin{enumerate}
        \item The map $x\mapsto g(x+c)$ from $\nu(\widehat{M})$ to $\nu_{g(Y)}(c)(\widehat{M})$ induces on $\nu(\widehat \CM)$ a $\CD^1$-group structure.
        \item Let $\alpha:(\CF,+)\to (\CF,+)$ be an $\CM$-definable endomorphism leaving the type $\nu$ invariant (namely, $\alpha_*(\nu)=\nu)$).
        Then $\alpha:\nu(\widehat \CM)\to \nu(\widehat \CM)$ is a $\CD^1$-map with respect to the above differential structure on $\nu$.
        \item For every $a\in \CF$, the function $\lambda_a:x\mapsto a\cdot x$ is differentiable at $0$ with respect to the above $\CD^1$-structure on $\nu$.  
        %namely $\sigma\lambda_a\sigma^{-1}$ is $K$-differentiable at $\sigma(0)$ (note that this can be formulated as a first order property of elements of $\CF$). 
    \end{enumerate}
\end{proposition}

\begin{proof}
%    Let $Y$, $g$ and $d$ be as in the statement.
All group operations appearing in the proof are the restriction to $Y$ of $\CF$-addition (and subtraction).
Let us explain what we mean by (1): Let $\sigma: \nu(\widehat{M})\to \nu_{g(Y)}(c)(\widehat{M})$ be given by $x\mapsto g(x+c)$. Consider the push-forward of $x+y$ via $\sigma$. Namely, the function on  $\nu_{g(Y)}(c)\times \nu_{g(Y)}(c)$ given by
\[(z,w)\mapsto g((g^{-1}(z)-c)+(g^{-1}(w)-c)+c)=g(g^{-1}(z)+g^{-1}(w)-c).\]
Similarly, consider the push-forward of $x\mapsto -x$, given by
\[z\mapsto g(-(g^{-1}(z)-c)+c).\]
We claim that both functions are $\CD^1$ in $K$, in a neighborhood of $(g(c),g(c))\in K^2$ and $g(c)\in K$, respectively.

Let $d\in Y$ be such that $\dpr(d/A)=\dpr(Y)$. By replacing $Y$ with $Y-d$, we may assume that $\nu \vdash Y$. Now absorb $A$ and $d$ into the language. In order to keep notation simple we identify $Y$ with its image under $g$ (so $g=\id$).
%Since $\nu$ is a subgroup, we may assume, using compactness and shrinking $Y$ if needehttps://www.overleaf.com/project/60265da308bbda2bba558baad, that $Y=-Y$ and that $Y'+Y'\sub Y$ for some definable $Y'$ of full rank.

\noindent (1)
% Note that after the above identification, $\nu=\{U\subseteq K^n: U \text{ is $K$-open and } 0_{\CF}\in U\}$ and thus $\nu(\widehat \CM)$ is open in $K(\widehat \CM)^n$.
        Since $\nu$ is a group, type definable over $\CM$, there are, by compactness, $\CM$-definable open sets $V_1, V_0$, such that $\nu\vdash V_1\subseteq V_0\sub K^n$ and
        %Choose any $V_0(\CK)=g^{-1}(U_0(\CK))-d$, where $g(d)\in U_0(\CK)\subseteq \CK^n$ is an $K$-definable open set. By compactness, there exists $K$-definable open $g(d)\in U_1(\CK)\subseteq U_0(\CK)$  such that for $V_1(\CK)=g^{-1}(U_1(\CK))-d$,
        \[\phi_4:=(x,y) \mapsto x+y \text{ maps } V_1^2 \text{ into } V_0.\]
        Similarly, we find $V_2\subseteq V_1$ such that
        \[\phi_3:= (x,y,z)\mapsto (-z,x+y) \text{ maps } V_2^3 \text{ into } V_1^2.\] We also find
        $V_3\subseteq V_2$ such that
        \[\phi_2:= (x,y,z)\mapsto (x+y,z) \text{ maps } V_3^3\text{ into } V_2^2,\]
        and $V_4\subseteq V_3$ such that
        \[\phi_1:= (x,y,z,w)\mapsto (w+x,-y,z) \text{ maps } V_4^4\text{ into }  V_3^3.\]
        We may assume that all the above are $\emptyset$-definable. Let $a,b\in V_4$ with $\dpr(a,b)=2\dpr(Y)=2\dpr(V_4)$. By {\gendif}, $x+y$ is $\CD^1$ at $(a,b)$ and the function $y\mapsto -y$ is $\CD^1$ at $a$, hence $\varphi_1(x,y,z,b)$ is $\CD^1$ at $(a,a,a)$. Similarly, $\varphi_2$ is $\CD^1$ at $(b+a,-a,a)$, $\varphi_3(x,y,b)$ at $(b,a)$ and $\varphi_4$ at $(b^{-1},b+a)$. Composing, we obtain $\varphi_4\circ\varphi_3\circ\varphi_2\circ\varphi_1(x,y,z)=x-y+z$, so we have shown that the map $(x,y,z) \mapsto x-y+z$ is $\CD^1$ at $(a,a,a)$.
        In fact, the proof provides an open set, $U\sub V_4$ which, by compactness, we may take to be $\CM$-definable, with $a\in U$, such that $(x,y,z)\mapsto x-y+z:U^3\to V_0$ is $\CD^1$.

        Our goal is to show that the push-forward of addition restricted to $\nu^2$ and of the inverse function restricted to $\nu$ via the map $x\mapsto x+a$ are $\CD^1$ (in the sense of $K$). Namely, we need to prove that the functions $(x-a)+(y-a)+a$ and $-(x-a)+a$ are  $\CD^1$ on $\nu_Y(a)^2$ and $\nu_Y(a)$, respectively. This follows immediately from our choice of $U$.
%\todo{As proved right now c=d+a Yatir: The prop says there exists a $c\in Y$, who is this c in the langauge of the proof.?}
%\todo{please explain the above comment}

        \noindent (2) 
        Let $\alpha$ be a an $\CM$-definable endomorphism of $(\CF,+)$ fixing $\nu$ setwise. Fix $c\in M$ as in (1) and $\sigma(x)=g(x+c)$.
        Choose $e\in \nu(\widehat \CM)$ with $\dpr(e/\CM)=\dpr(\nu)=n$. By {\gendif}, $\alpha$ is $\CD^1$ at $e$ with respect to the differentiable structure on $\nu(\widehat \CM)$ (i.e., $\sigma\alpha \sigma^{-1}$ is $\CD^1$ at $\sigma(e)$).  Since $\alpha$ is a homomorphism and $\nu$ is a $\CD^1$-group, it is standard to verify that $\alpha$ is a $\CD^1$-function on all of  $\nu$.

       \noindent (3) By Proposition \ref{P:infinit}, for every $a\in \CF(\CM)$, $\lambda_a$ is a endomorphism of $\nu$, so by (2), it is $\nu$- differentiable at $0$, i.e., $\sigma\lambda_a\sigma^{-1}$ is $K$-differentiable at $c=\sigma(0)$. %This is a first order property of $a\in \CF$, and therefore for every $a\in \CF(\widehat \CM)$, $\sigma\lambda_a\sigma^{-1}$ is $K$-differentiable at $c$.
\end{proof}

\subsection{Strong internality to SW-fields}
The next theorem is an important step in our proof of Theorem \ref{T: main interp}. First, we need the following easy and well known fact:
\begin{remark}\label{remarkk}
    If $\CL=(L,+,\cdot,\dots)$ is a dp-minimal expansion of a field, then $L$ has no definable infinite subfields. Indeed, if $K$ were such a field then $K$ itself is dp-minimal. And if we had some $u\in L\setminus K$ we could define $T:K^2\to L$ by $(a,b)\mapsto a+bu$. Since $u$ is $K$-linearly independent of $1$, we get that $T$ is a linear injection, so $\dpr(T(K^2))=2$ which is impossible.
\end{remark}

\begin{theorem}\label{prop-main}
 Let $\CM$ be a multi-sorted structure and $K$  an $\CM$-definable SW-uniform field satisfying {\gendif}.  Let $\CF$ be a definable field of finite dp-rank that is locally strongly internal to $K$. Then $\CF$ is definably isomorphic to a finite extension of $K$.
\end{theorem}
\begin{proof}
    By passing to an elementary extension, we may assume that $\CM$ is $|T|^+$-saturated.
        Let $Y\subseteq \CF$ be $K$-critical, assume that $Y$ (and a corresponding definable injection  $g:Y\to K^n$) is defined over some small model $\CN$.  By \cite[Proposition 4.6]{SimWal}, we may assume that $g(Y)$ is open. Let $\nu$ be the infinitesimal subgroup of $(\CF,+)$, as given by Proposition \ref{P:infinit}, endowed with its Lie group structure. It will be convenient to evaluate $\nu$ in some $|\CM|^+$-saturated $\widehat \CM\succ \CM$.

        By Proposition \ref{P:C1 structure}(3), for every $z\in \CF(\CM)$, the function $\lambda_z$ is $\CD^1$ at $0$, with respect to this differential structure on $\nu$.
        
 To each $z\in \CF$ we associate, definably, the differential $D_0(\lambda_z)$ identified with the Jacobian matrix  of $\lambda_z$ at $0$, with respect to the differentiable structure on $\nu$ (formally the Jacobian matrix of $\sigma\lambda_z\sigma^{-1}$ at $\sigma(0)$).

We claim that the function $z\mapsto J_z$ is a ring homomorphism from $(\CF,+,\cdot)$ into $(M_n(K),+,\cdot)$. To see that we shall apply the chain rule to our differentiable functions
(see \cite[Chapter 7.1]{vdDries} for the real closed case and \cite[Remark 4.1.ii]{Schneider} for the valuation case\footnote{Note that in the proof of the latter, one must replace the usual operator norm with the $\inf$-valuation on the matrix space.}).
We now recall the arguments from  \cite[Lemma 4.3]{OtPePi}:

To see that field multiplication is sent to matrix multiplication:
First, note that $\lambda_{a\cdot b}=\lambda_a\circ \lambda_b$, and hence by the chain rule, 
$$D_0(\lambda_{a\cdot b})=D_0(\lambda_a\circ\lambda_b)=D_0(\lambda_a)\cdot D_0(\lambda_b), $$ where on the right we have matrix multiplication.

To see that field addition is sent to matrix addition, let $P(x,y)=x+y$ and note that  $\lambda_{a+b}=\lambda_{P(a,b)}=P(\lambda_a,\lambda_b)$. It is easy to see that  $D_{(0,0)}(P)=(I_n,I_n)$, where $I_n$ is the $n\times n$ identity matrix
(since $P(x,0)=P(0,x)=x$).  By the chain rule,
\[D_0(\lambda_{a+b})=D_0(\lambda_{P(a,b)})=D_{(0,0)}(P\circ(\lambda_a,\lambda_b))=D_0(\lambda_a)+D_0(\lambda_b),\] where on the right we have matrix addition.
      
      Thus,  the map $z\mapsto D_0(\lambda_z)$ is a ring homomorphism  sending $1\in \CF$ to the identity matrix $I_n$. Since $\CF$ is a field, the map is injective so we have definably embedded $\CF$ into a definable subring of $M_n(K)$.

        We may now view $\CF$ as a definable subfield of $M_n(K)$. Let $K_0=\{xI_n:x\in K\}$, where now we take the usual scalar multiplication in the algebra of matrices. Note that $K_0\cap \CF$ is an infinite definable subfield of $K$. Indeed, if the characteristic of $K$ is $0$ then it follows since both contain $I_n$. If the characteristic is positive then it follows since both contain $\mathbb{F}_p^{alg}$ by \cite[Corollary 4.5]{KaScWa}. Since $K_0\cong K$ is dp-minimal, it has no infinite definable subfield by Remark \ref{remarkk}, so $K_0\cap \CF=K_0$ i.e. $K_0\subseteq \CF$. Thus $\CF$ is a finite extension of $K_0$.
    \end{proof}

\subsection{1-h-minimal valued fields}
In  \cite{hensel-min,hensel-minII} Cluckers, Halupczok, Rideau-Kikuchi and Vermeulen introduce the class of  \emph{1-h-minimal valued fields} in characteristic 0, encompassing several examples of interest, such as:

\begin{example}\cite[Corollaries 6.2.6,6.2.7]{hensel-min}\cite[Theorem 6.3.4]{hensel-min}\cite[Proposition 6.4.2]{hensel-min}
\begin{enumerate}
\item pure henselian valued fields of characteristic $0$.
\item finite field extensions of $\mathbb{Q}_p$ in the sub-analytic language
\item henselian valued  fields of characteristic $0$ in the valued field language expanded by function symbols from a separated Weierstrass system $\mathcal{A}$ and equipped with analytic $\mathcal{A}$-structure.
\item $T$-convex valued fields expanding a power-bounded o-minimal field.
\item $V$-minimal fields.
\end{enumerate}
\end{example}

The exact definition of 1-h-minimal fields is irrelevant for the present section (see Section \ref{ss:1-h-minimality and its connections}). What will be relevant for us here is that they satisfy {\gendif} (see below).

\begin{corollary}\label{C:field in dp-min h-min}
Let $\CM$ be a multi-sorted structure and $\CF$ a definable field of finite dp-rank. If $\CF$ is locally strongly internal to a definable dp-minimal 1-h-minimal valued field $(K,v)$ then it is definably isomorphic to a finite extension of $K$. In particular, this is true if $K$ is a  pure dp-minimal valued field of characteristic $0$.
\end{corollary}
\begin{proof}
Every dp-minimal expansion of a valued field is an SW-uniformity by Example \ref{E:SW-uniformities}.  Since $\CF$ is 1-h-minimal, it satisfies {\gendif} by \cite[Theorem 5.1.5]{hensel-min} and \cite[Proposition 3.1.1]{hensel-minII}. We may now apply Theorem \ref{prop-main} with $\CF$ the definable field. Since pure henselian valued fields are 1-h-minimal (by Clause (1) of the above example), the conclusion of the addendum follows. 
\end{proof}

\subsection{Strong internality to strongly minimal fields}\label{ss:strongly internal to strongly minimal}
In Theorem \ref{prop-main} we  classified, under the additional assumption of generic differentiability,  fields locally strongly internal to  SW-uniform fields. Such fields are clearly unstable.  We now turn our attention to  fields  locally strongly internal to a strongly minimal field. In the following let $\mr$ be the Morley rank and $\dM$ be the Morley degree.
%\todo{Isn't it a little strange to have RM and dM, rather than DM? Yatir: I don't mind. Assaf?}

\begin{proposition}\label{P:SI to SM}
Let $\CM$ be multi-sorted structure and let $\mathscr{K}$ be a strongly minimal definable field. If $\CF$ is a definable field of finite dp-rank that is locally strongly internal to $\mathscr{K}$ %and $\mathscr{K}$ eliminates imaginaries
then $\CF$ is strongly internal to  $\mathscr{K}$, so in particular it is algebraically closed. If  $\mathscr{K}$ is a pure field  then $\CF$ is definably isomorphic to $\mathscr{K}$.
\end{proposition}
\begin{proof}
Since $\mathscr{K}$ is strongly minimal, it is dp-minimal and acl satisfies exchange. As a result, dp-rank is equal to the acl-dimension that is exactly Morley rank. Furthermore, by e.g. \cite[Lemmas 8.4.10, 8.4.11]{TZ}, since $\acl(\emptyset)$ is infinite, $\mathscr{K}$ eliminates imaginaries.

Among all definable subsets of $\CF$ strongly internal to $\mathscr{K}$ choose one, call it $Y$, that is $\mathscr{K}$-critical.  Translating $Y$, if needed, we may assume that $0_\CF\in Y$. Since $Y$ is in definable bijection with a definable subset of $\mathscr{K}^n$ then, after shrinking $Y$, we may assume that $Y$ is of finite Morley rank and  $\mathrm{dM}(Y)=1$. We also need the observation that if $D_1,D_2\sub \CF$ are definable subsets strongly internal to $\mathscr{K}$ then so is $D_1\times D_2$, and since $\mathscr{K}$ eliminates imaginaries, so is $(D_1\times D_2)/E$ for any definable equivalence relation on $D_1\times D_2$. In particular, as $\mathscr{K}$ eliminates imaginaries, $D_1+D_2=\{g+h:g\in D_1,\, h\in D_2\}$ is strongly internal to $\mathscr{K}$ (operations are taken in $\CF$).

What follows is, essentially, a local version of Zil'ber's indecomposability theorem, not assuming (a priori) an ambient $\omega$-stable group.

We work inside $Y-Y+Y\supseteq Y$. It is $\omega$-stable (since it is strongly internal to $\mathscr{K}$), and since $Y$ is $\mathscr{K}$-critical  $\dpr(Y-Y+Y)=\dpr(Y)$. So, in fact, $\mr(Y-Y+Y)=\dpr(Y-Y+Y)=\dpr(Y)=\mr(Y)$.
%(where we view $Y$ as a definable subset of $Y-Y+Y$).

Let $R=\{(g,h)\in Y^2: \mr((g+Y)\triangle (h+Y))<\mr(Y)\}$. The set $R$ is definable since $\mathscr{K}$ is strongly minimal and $Y\sub \mathscr{K}^n$ (see \cite[Lemma 6.2.20]{MaBook}). Moreover, it is an equivalence relation on $Y$ (since $\mathrm{dM}(Y)=1$).
%Since $\dM(Y)=1$, $(g,h)\in R$ is equivalent to $\mr((g+Y)\cap (h+Y))=\mr(Y)$. It is obviously reflexive (since $Y$ is infinite) and symmetric. It is also %transitive: assume that $(g,h),(h,r)\in R$, i.e. $\mr((g+Y)\cap(h+Y))=\mr((h+Y)\cap (r+Y))=\mr(Y)$. It now follows readily that $\mr((g+Y)\cap %(r+Y))=\mr(Y)$ as well.

Let $k=\dM(Y+Y)$. We claim that there are at most $k$-many $R$-equivalence classes in $Y$. For  assume that there were at least $k+1$ classes represented by $g_1, \dots, g_{k+1}$. Since $(\mr, \dM)(g_i+Y)=(\mr, \dM)(Y)=(\mr(Y), 1)$, by the definition of $R$ this means  that $\mr(g_i+Y\cap g_j+Y)<\mr(Y)$. Since $g_i+Y \sub Y+Y$ for all $i$, it follows that either $\dM(Y+Y)>k$ or $\mr(Y+Y)>\mr(Y)$. Either option is impossible.

Hence there are only finitely many $R$-classes in $Y$ and since they cover $Y$  one of them has Morley rank $\mr(Y)$. So there is a definable subset $Y_1\subseteq Y$ with $\mr(Y_1)=\mr(Y)$ satisfying that for any $g,h\in Y_1$, $\mr((g+Y)\cap (h+Y)=\mr(Y)$.

Let $p$ the unique generic type of $Y$ and observe that for $g\in Y$, $g+p$ is the unique generic type of  $g+Y$. Let $H=\mathrm{Stab}(p)=\{a\in \CF:a+p=p\}$. Note that, if $a\in H$ then $a+p=p$ so that $(a+Y)\cap Y$ is generic in $Y$, and in particular  $H\subseteq Y-Y\subseteq Y-Y+Y$. Hence $H$ is a type-definable group in $\mathscr{K}$, and therefore $H$ is, in fact, definable by e.g. \cite[Theorem 7.5.3]{MaBook}. By the definition of $Y_1$ above, we have $Y_1-Y_1\sub H$, hence
\[
    \mr(Y)=\mr(Y_1-Y_1)\le \mr(H)\le \mr(Y-Y)=\mr(Y).
\]
So $\mr(H)=\mr(Y)$.

Replace $H$ by $H^0$, its connected component, and let $q$ be its generic type. We claim that $H$ is invariant under $\CF$-multiplication. Let $c\in \CF$. We work now in $cH+H$, that is still strongly internal to $\mathscr{K}$. If $cH\not\subseteq H$ then $H/(cH\cap H)>1$. Since $cH$ is connected as well it must be infinite. On the other hand $H/(cH\cap H)\cong (cH+H)/H$, so if the latter is infinite $\mr(cH+H)>\mr(H)=\mr(Y)$, contradicting the assumption that $Y$ is $\mathscr{K}$-critical.

As a result, $H$ is a non-zero ideal of $\CF$ i.e. $H=\CF$ and hence $\CF$ is an $\omega$-stable field so algebraically closed. If $\mathscr{K}$ is pure then  by \cite{PoiFields}, $H$ is definably isomorphic to $\mathscr{K}$.
\end{proof}

At the level of generality we are working in the statement of Proposition \ref{P:SI to SM} is optimal: \begin{example}
    Let $\mathcal K\models$ ACVF and $\bk$ its residue field. Expand $\bk$ by fusing it (in the sense of \cite{Hr2}) with an algebraically closed field $\CF$ of a different characteristic. By (the proof of) \cite[Propostion 5.9]{HaHa} the resulting expansion of $\CK$ is still  dp-minimal, and $\bk$ is still stably embedded in the expanded structure. So $\CK$ is an SW-uniformity.  However, the field $\CF$ (with its induced structure) is not definably isomorphic to the field $\bk$ (with its expanded structure).
\end{example}

As a corollary we obtain another isomorphism result.

\begin{corollary}\label{dp minimal fields}
Let $\CM$ be a multi-sorted structure and $K$ an infinite dp-minimal pure field of characteristic $0$ definable in $\CM$. If $\CF$ is a field of finite dp-rank which is locally strongly internal to $K$ then $\CF$ is definably isomorphic to  a finite extension of $K$. In particular, any field $\CF$ definable in a pure dp-minimal field $K$ of characteristic $0$ is definably isomorphic to a finite extension of $K$. 
\end{corollary}
\begin{proof}

By Johnson's theorem (Fact \ref{F:Johnson DP}), $K$ is either algebraically closed, real closed or admits a definable henselian valuation. If $K$ is algebraically closed, the result follows from Proposition \ref{P:SI to SM}.
The remaining cases follow from Theorem \ref{prop-main} and the fact that definable functions in a pure real closed field or a pure dp-minimal valued filed satisfy {\gendif} (by o-minimality in the former case and by 1-h-minimality in the latter case).  
\end{proof}

\subsection{From a finite-to-finite correspondence to strong internality}
The results above were all proved assuming the existence of a definable injection from an infinite subset of $\CF$ into $K^n$.
As we note here, the injectivity assumption can often be relaxed to a finite-to-finite correspondence.
Let us, first, clarify our terminology: 
\begin{definition}
    Let $X,Y$ be any sets. A relation $C\sub X\times Y$ is a \emph{finite-to-finite correspondence between $X$ and $Y$} if the projections  $\pi_1: C\to X$ and $\pi_2:C\to Y$ are surjective with finite fibres. 
\end{definition}

We also fix some notation, for the following standard notion: 

\begin{efbi}
A definable set $D$ in a (multi-sorted) structure $\CM$ has
\emph{elimination of finite imaginaries} if whenever $\{X_t\subseteq D^n:t\in T\}$ is a definable family of finite sets, uniformly bounded in size  then there exists a definable map $f:T\to D^m$ for some integer $m$, with the property that $f(t_1)=f(t_2)$ if and only if $X_{t_1}=X_{t_2}$.
\end{efbi}

The condition {\efi} is satisfied whenever $D$ expands a definable field (using symmetric functions) or a linear order. It  allows us, under additional assumptions, to replace definable finite-to-finite correspondences by definable bijections:

\begin{lemma}\label{L:almost internal-to -strongly}
Assume that $X$ and $Y$ are definable in some $|T|^+$-saturated (multi-sorted) structure $\CM$ and   there is a finite-to-finite definable correspondence between infinite subsets of $X$ and $Y$.

Assume also that
\begin{enumerate}
    \item $X$ has {\efi}, and 
    \item either $Y$ supports an SW-uniform structure or $Y$ has {\efi}.
\end{enumerate}
Then $X$ is locally strongly internal to $Y$.
\end{lemma}
\begin{proof}
Restricting $X$ and $Y$ we may assume that there is  $C\sub X\times Y$, a definable finite-to-finite correspondence between $X$ and $Y$.  For  $x\in X$ and $y\in Y$ denote
$$C_x=\{y\in Y:(x,y)\in C\},\, C^y=\{x\in X:(x,y)\in C\}$$ and note that by $\aleph_0$-saturation, they are uniformly bounded in size by some integer.

We first claim that there is a finite-to-one function from an infinite subset of $Y$ into $X$.
Since $X$ satisfies {\efi}, each $C^y$ is coded by some element of $X^m$, for some integer $m$, thus we obtain a definable finite-to-one function $f$ from a $\pi_2(C)\sub Y$ into $X^m$. Amongst all definable finite-to-one functions from an infinite definable subset of $Y'\subseteq Y$ into $X^m$, choose one, $h$, with $m$ minimal. We claim that necessarily $m=1$.

Otherwise, consider the projection
$\pi:X^m\to X^{m-1}$ onto the first $m-1$ coordinates and let  $W=\pi(X')$, for $X'=h(Y')$.
If there is some $w\in W$ such that $\pi^{-1}(w)$ is infinite then we obtain a finite-to-one map from an infinite subset of $Y$ into $\pi^{-1}(w)$ so also into $X$, contradicting the fact that $m>1$

Otherwise  $\pi^{-1}(w)$ is finite for all $w\in W$. The function $h_1=\pi\circ h$ is again finite-to-one from $Y'$ into $W\sub X^{m-1}$, contradicting the minimality of $m$.

We thus showed the existence of a definable finite-to-one $h:Y'\to X$, for some infinite definable $Y'\sub Y$. Without loss of generality, $Y'=Y.$

Assume that $Y$ is an SW-uniformity, and all the data is definable over $A$.
Let $c\in Y$ be with $\dpr(c/A)=1=\dpr(Y)$. Since the topology on $Y$ is Hausdorff, we can find a relatively open subset $U\subseteq Y$ with
$c\in U$ such that $h^{-1}(h(c))\cap U=\{c\}$. By Proposition \ref{P:genos} there exists some $B\supseteq A$ and a $B$-definable open neighborhood $U_0\subseteq U$ of  $c\in U_0$ such that $\dpr(c/B)=\dpr(c/A)=1$.

 Now let $Y_0=\{y\in U_0:|h^{-1}(h(y))|=1\}$. By the above, $Y_0$ is $B$-definable. As $c\in S$ and $\dpr(c/B)=1$, $Y_0$ must be infinite
 and hence $h\restriction Y_0$ is injective.

 Finally, assume that $Y$ has {\efi}. Each of the fibers of the map  $h$ can be coded by an element of $Y^k$ for some fixed $k$, so we obtain a definable injection from an infinite subset of $X$ into $Y^k$.
\end{proof}

\begin{corollary}\label{C:finite-to-finite enough for interp}
Let $\CM$ be a multi-sorted structure, $(K,+,\cdot)$ an SW-uniform field satisfying {\gendif} and $\CF$  a definable field of finite dp-rank. If there exists an infinite definable $S\subseteq\CF$ and a finite-to-finite definable correspondence from $S$ into $K^n$ then $\CF$ is definably isomorphic to a finite extension of $K$.
\end{corollary}
\begin{proof}
Let $f$ be the definable correspondence from $S$ into $K^n$. As $K$ is dp-minimal it eliminates $\exists^{\infty}$ by \cite[Lemma 2.2]{DoGoStrong} and we may thus assume $K$ is $|T|^+$-saturated. Both $\CF$ and $\CK$ are fields so satisfy {\efi}. By Lemma \ref{L:almost internal-to -strongly}, $\CF$ is locally strongly internal to $K$. We are now in the situation to apply Theorem \ref{prop-main} and conclude the proof.
\end{proof}

\section{Reduction to the distinguished sorts}\label{S:the distinguished sorts}

\newcommand{\CB}{\mathscr{B}}
Throughout this section, and until the end of the paper we set:
\begin{assumption}
Let  $\CK=(K,v,\dots)$ denote a dp-minimal expansion of a valued  field.
%with all other sorts in $\CK^{eq}$. Throughout $\CF$ denotes an infinite field interpretable in $\CK$.\todo{I don't like the assumption that the structure is just %$K^{eq}$. It reduces the generaliy for no apparent reason.}
%\todo{Kobi:I think it can be omitted in sec 5(in 5.4 we talk about $K^n/E$)}
%\todo{Yatir: I agree. We only need to make it clear in 5.13 that K is one-sorted, I think. If so - we might need to change that line in the preliminaries.}
\end{assumption}

%\todo{Let's change the paragraph below, somehow it does not seem informative to me}

In this section we aim to show that if $\CK$ is either $P$-minimal, $C$-minimal or power bounded $T$-convex then every infinite field $\CF$ interpretable in $\CK$ is locally strongly internal to one of the distinguished sorts, $K$, $\Gamma$, $\bk$ or $K/\CO$. In practice, at this level of generality, we show a weaker result (replacing the definable bijection in the definition of strong internality  with a finite-to-finite correspondence), still sufficient for our needs. We start with a general observation assuring that $\CF$ is locally strongly internal to a unary imaginary sort:

\begin{lemma}\label{1-dim}
Let $\CM$ be an arbitrary $\aleph_0$-saturated structure, and $E$ a definable equivalence relation on $X\sub M^k$ with infinitely many classes. Then there exists an infinite definable quotient $X'/E'$, with $X'\sub M$ and a  finite-to-finite definable correspondence  between $X'/E'$ and an infinite subset of   $X/E$.
\end{lemma}
\begin{proof}
We fix $n$ minimal with respect to the following property:  There is a finite-to-finite definable correspondence between $X'/E'$ and an infinite subset of $X/E$, such that $X'\sub M^n$.
 Our goal is to prove that that $n=1$.
    
    Thus we may already assume that $X\sub M^n$ for this minimal $n$, and towards contradiction we assume that $n>1$.
    Let $\pi: X'\to M^{n-1}$ be the projection onto the first $n-1$ coordinates. Let $W=\pi(X')$. If for some $w\in W$ the set $X'_w/E'$ is infinite, where $X'_w:=X'\cap \pi^{-1}(w)$, we get a contradiction to  the minimality of $n$ (since we can definably identify $X'_w$ with a subset of $M$ and then obtain a one-to-one map from $X'_w/E'$ into $X/E$).

So  we assume that $|X'_w/E|$ is finite for all $w\in W$. As $\CM$ is $\aleph_0$-saturated, Without loss of generality, we may assume that for some fixed $s\in \mathbb N$, for all $w\in W$, $|X_w/E|=s$.
\vspace{.1cm}

We now define on $W$ the equivalence relation $$w_1 E_1 w_2 \Leftrightarrow X_{w_1}/E=X_{w_2}/E.$$

Because $X/E$ is infinite and each $X_{w_1}/E$ is finite, the quotient $W/E_1$ is infinite. For $\alpha\in W/E_1$ we let $X_\alpha:=X_w/E$, for any $w\in \alpha$. By our assumption, $|X_\alpha|=s$ for all $\alpha\in W/E_1$. Notice that if $s=1$ then, by the definition of $E_1$, we get an injection from $W/E_1$ into $X/E$, contradicting the minimality of $n$ (since $X_1\sub M^{n-1}$). We assume then that $s>1$.

\noindent {\bf Case 1.} For every $\beta \in X/E$, $\beta$ belongs to at most finitely many (finite) sets of the form $X_\alpha$, $\alpha\in W/E_1$.
\vspace{.1cm}

In this case, the set $$C=\{(\alpha, \beta)\in W/E_1\times X/E:\beta \in X_w/E\}$$
is a finite-to-finite definable correspondence between $W/E_1$ and $X/E$, contradicting the minimality of $n$.

\vspace{.1cm}

\noindent {\bf Case 2.} There exists $\beta\in X/E$ which belongs to infinitely many sets of the form $X_\alpha$,  $\alpha\in W/E_1$.

\vspace{.1cm}

Fix such a $\beta$ and consider the infinite definable set $X^\beta=\{w\in W:\beta\in X_w/E\}$. By our assumption, $X^\beta/E_1$ is infinite. So we may replace $W$ by $X^\beta$ and assume that for each $\alpha\in W/E_1$, we have $\beta\in X_\alpha$.

We now replace $X$ by $X_1=X\setminus \beta$ (namely we remove one $E$-class, $\beta$, from $X$). Since we assumed that for every $\alpha$, $|X_\alpha|=s>1$ and $\beta\in X_\alpha$, the set $(X_1)_\alpha=(X_1)_w/E$, for $w\in \alpha$ has one element less, namely consists of $s-1$ elements.
By repeating the process we can finally reach the situation where $|X_\alpha|=1$, for every $\alpha$ and finish as above.\end{proof}

From now on we focus our attention on quotients of the form  $K/E$,  for some definable equivalence relation $E$ on $K$.

 %We split our discussion into two parts. The first deals with definable families of \emph{finite} sets of balls, whereas the second has no restriction on the families of balls, but assumes the valued field to be  ordered and the valuation ring to be convex.

\subsection{Inter-algebraicity with the distinguished sorts}\label{ss:unordered}

%{\bf Notation}
%\vspace{.2cm}

We first introduce some notation.
We let $\CB^{op}_\gamma$ be the set of  open balls in $K$ of radius $\gamma$ and
 $\CB^{cl}_\gamma$  the set of  closed balls in $K$  of radius $\gamma$. These are clearly definable families of sets and thus as we vary $\gamma\in \Gamma$ we obtain $\CB^{op}$ and $\CB^{cl}$ the families of  open and closed balls, respectively. We let $\CB=\CB^{op}\cup \CB^{cl}$ be the family of all balls. When $\Gamma$ is discrete  every closed ball is also open and thus $\CB^{op}=\CB^{cl}$.

Note that for any $\gamma_1,\gamma_2\in \Gamma$, there is a definable bijection (possibly using an additional parameter) between the set $\CB^{op}_{\gamma_1}$ and $\CB^{op}_{\gamma_2}$ and similarly between $\CB^{cl}_{\gamma_1}$ and $\CB^{cl}_{\gamma_2}$. In particular, there is a definable bijection between every $\CB^{cl}_\gamma$ and $K/\CO$, as the latter is just $\CB^{cl}_0$. Similarly, there is a definable injection of $\bk$ into $\CB^{op}_0$.

Given a set $X\sub K$, {\em a maximal ball} inside $X$ is a ball $b\subseteq X$ such that there does not exist a ball $b'$, $b\subsetneq b'\subseteq X$. Both $b$ and $b'$ can be either closed or open. As for any two balls $b_1,b_2$, either $b_1\cap b_2=\0$ or one of the balls is contained in the other, any two maximal balls in $X$ are necessarily disjoint. 
%Note that if $b$ is a maximal ball inside $X$ then it is the maximal ball containing any element of $b$.

%For any two balls $b_1,b_2$, either $b_1\cap b_2=\0$ or one of the balls is contained in the other.
By the same observation as above, if $X\sub K$ is definable in a structure expanding a valued field and $x_0$ is an interior point of $ X$, then the family of balls $b\sub X$ containing $x$  is a definable chain of balls. If the induced structure on $\Gamma$ is o-minimal (or, more generally,  definably complete) then the infimum of the radii of these balls exists and thus there is a maximal ball $b$ containing $x_0$ in $X$  (which could be closed or open, depending on whether this infimum is attained or not). The family of all maximal balls in $X$ is thus a definable family and if $\Gamma$ is definably complete then its union covers $X$.

In order to study uniformly definable finite sets of balls,   the following additional assumption  regarding the valued field $K$ will be useful.

\begin{min-balln}
  For every $\widehat{\mathcal{K}}\equiv \mathcal K$, if $X\subseteq \widehat K$ is a definable subset  intersecting infinitely many closed $0$-balls then it contains a closed ball of radius $<0$.
\end{min-balln}

\begin{remark}

\begin{enumerate}
    \item It is easy to verify that if $\CK$ satisfies {\minballn} then it satisfies the same statement with $0$ replaced by any $\gamma_0\in \Gamma$.
    \item  Under our assumptions, if $\bk$ is finite then $\Gamma$ is  discrete (by Johnson, \cite[Lemma 4.3.1]{johnsonPhD}, see also \cite[Lemma 2.7]{HaHa} for a concise  overview). Assuming {\minballn}, if $\Gamma$ is discrete then $\bk$ is finite. Indeed, if $\bk$ were infinite then the definable set $\CO\setminus \{0\}$ would contain infinitely many closed balls of radius $1$, but no closed ball of radius $0$, contradicting {\minballn}.
\end{enumerate}
 \end{remark}

%\begin{bnd-openn}
% For every $\widehat{\mathcal{K}}\equiv \mathcal K$, if $X\subseteq \widehat K$ set which intersects  infinitely many open %$\gamma_0$-balls then it contains an open ball of radius $\gamma_0$.
%\end{bnd-openn}

\begin{lemma}\label{fin-max} 
If $\CK$ satisfies {\minballn} and $X\sub K$ is definable then for every $\gamma_0\in \Gamma$,  $X$ contains at most finitely many maximal closed balls of radius $\gamma_0$.\end{lemma}
\begin{proof} 
Assume towards a contradiction that  $X$ contained infinitely many closed maximal $\gamma_0$-balls. Let $X'\sub X$ be the union of all the closed maximal $\gamma_0$-balls in $X$. The set $X'$ is definable and each of these closed $\gamma_0$-balls is still maximal in $X'$. By {\minballn}, $X'$ must contain a closed ball of radius $<\gamma_0$. But then one of the $\gamma_0$-balls in $X'$ is not maximal.
\end{proof}

Recall that $\Gamma$ has {\em Definable Skolem Functions} if given a definable family $\{X_t:t\in T\}$  of non-empty subsets of $\Gamma$, there is a definable function $c:T\to \Gamma$ such that $c(t)\in X_t$. The following definition is based on \cite{PilPoi}.
\begin{definition} We say that a structure $\CM$ is {\em surgical} if every definable equivalence relation on $M$ has at most finitely many infinite classes.
\end{definition}
\begin{proposition}\label{new-main}
Assume that $\CK=(K,v,\ldots)$ is a sufficiently saturated expansion of a dp-minimal valued field satisfying 
\begin{enumerate} 
\item {\minballn}.
\item  $\Gamma$ is definably complete and has definable Skolem functions.
\item $\bk$ is surgical.
\end{enumerate}

For every definable  $X\subseteq K^n$ and definable equivalence relation $E$ on $K$ with $X/E$ infinite, there exists a definable infinite $T\sub X/E$ and a definable finite-to-finite correspondence between $T$ and a definable subset of $K$, $\Gamma$, $\bk$ or $K/\CO$.
\end{proposition}
\begin{proof} 
By Lemma \ref{1-dim}, there exists $T\sub X/E$ and a finite-to-finite definable correspondence between  $T$ and $K/E'$, for some definable equivalence relation $E'$ on $K$.
Thus it is sufficient to prove the proposition for  $X\subseteq K$. So, let $E$ be a definable equivalence relation on $X$ with $T=X/E$ infinite and for each $t\in T$ we let $E_t\sub X$ be the corresponding $E$-class.

If there are infinitely many $t\in T$ with $E_t$ finite, by passing to a definable subset of $T$ (and using the fact that $\CK$ eliminates $\exists^\infty$ by \cite[Lemma 2.2]{DoGoStrong}), we may assume that $E_t$ is finite for all $t\in T$. This gives a one-to-finite map between $T$ and $K$ as needed.

We now assume that $E_t$ is infinite for all $t\in T$.
By the definable completeness of $\Gamma$, each $x\in E_t$ is contained in a (unique) maximal ball inside $E_t$. Thus, $E_t$ can be written as a disjoint union of maximal sub-balls, and the map which assigns to each $t\in T$ the set $S_t\sub \CB$ of maximal sub-balls (open or closed) of $E_t$ is definable.

Since $\Gamma$ has definable Skolem functions, we can definably choose for each $t\in T$ a radius $r(t)\in \Gamma$ of one of the balls in $S_t$. By shrinking $X$ and $E$ (but not $K/E$)
we may assume that each $E_t$ is a union of maximal balls, all of fixed radius $r(t)$. \\

\noindent{\bf Case 1} The map $t\mapsto r(t)$ is finite-to-one. \\

This immediately yields a finite-to-one map from $T$ into $\Gamma$. A finite-to-one map is a finite-to-finite correspondence, so we are done. \\

\noindent{\bf Case 2} There exists $\gamma_0\in \Gamma$ such that $T'=r^{-1}(\gamma_0)$ is infinite. \\

By replacing $T$ with  $T'$ we may assume that for all $t\in T$, $E_t$ is a union of maximal balls, all of radius $\gamma_0$.
Using a definable bijection we may assume that $\gamma_0=0$,  and so all maximal balls in $E_t$ are of radius $0$.

Assume first that there are infinitely many $t\in T$ such that one of the maximal balls in $E_t$ is closed. By restricting $T$ we may assume that  for all $t\in T$, one of the maximal balls in $E_t$ is closed.
By Lemma \ref{fin-max}, each $E_t$ contains at most finitely many maximal closed $0$-balls and the map $t\mapsto S_t\sub \CB^{cl}_0=K/\CO$ sending $t\in T$ to the finite set of its closed maximal $0$-balls in $E_t$ is definable. Since the $E_t$ are pairwise disjoint this gives rise to a one-to-finite correspondence between  $T$ and an infinite subset of  $K/\CO$.

Consequently, we may now assume that for every $t\in T$, each maximal ball in $E_t$ is open of radius $0$.
By {\minballn}, each $E_t$ can intersect only finitely many closed $0$-balls (otherwise, it will contain a closed ball of radius $<0$, contradicting the maximality of all the $0$-balls).\\

\noindent{\bf Case 2.a} There exists a fixed closed $0$-ball $b_0$ intersecting infinitely many  classes $E_t$. \\

By translating, we may assume that $b_0=\CO$. By intersecting each $E_t$ with $\CO$ we may assume, by further shrinking $T$,  that every $E_t$ is contained in $\CO$.
 Thus, each $E_t$ is sent by $\mathrm{res}:\CO\to \bk$ into a subset of $\bk$ and for $t_1\neq t_2$, we have $\mathrm{res}(E_{t_1})\cap \mathrm{res}(E_{t_2})=\0$.
This induces a definable equivalence relation on $\bk$ so by our assumption, only finitely many of these classes are infinite. As $\bk$ is a field of finite dp-rank, by \cite[Lemma 2.2]{DoGoStrong} it eliminates $\exists^{\infty}$ so  by reducing $T$, we obtain a definable one-to-finite correspondence between $T$ and an infinite subset of $\bk$.\\

\noindent{\bf Case 2.b} Every closed $0$-ball intersects only finitely many classes $E_t$.
\\

The map sending $t\in T$ to the finite set $F_t\sub \CB^{cl}_{0}$ of closed $0$-balls  intersecting $E_t$ is definable. By our assumption, each ball $b\in F_t$ intersects at most finitely many of the $E_t$, thus the function $t\to F_t$ is finite-to-one, and we obtain a finite-to-finite correspondence between (infinite subsets of) $T$ and $K/\CO$.
\end{proof}

In the rest of this section we examine various settings in which the assumptions of Proposition \ref{new-main} hold.

\subsubsection{The C-minimal case}
Recall the definition of a C-minimal valued field from Section \ref{sss:c-minimal}.

\begin{proposition}\label{P:properties for Cmin}
Let $(K,v,\dots)$ be a C-minimal valued field. Then
\begin{enumerate}

\item $\Gamma$ is definably complete and has definable Skolem functions, and $\bk$ is surgical.
    \item $\CK$ satisfies {\minballn}.
\end{enumerate}

%As a result, every C-minimal valued field satisfies {\minball}.
\end{proposition}
\begin{proof}

(1) Follows from the o-minimality of $\Gamma$ and the strong minimality of $\bk$. 

(2) Assume that $X\sub K$ intersects infinitely many closed $0$-balls. We want to show that it contains a closed ball of negative radius. Since $X$ is a finite union of Swiss cheeses  one of them intersects infinitely many closed $0$-balls. So we may assume that $X$ is a Swiss Cheese, namely $X=b\setminus \bigcup\limits_{i=1}^s b_i$.

For every closed $0$-ball $b_0$  intersecting $X$, either $b\sub b_0$ or $b_0\sub b$. In the first case, $b_0$ is the only $0$-ball intersecting $b$ thus $b_0\subsetneq  b$. Note that there are only finitely many closed $0$-balls  intersecting $X$ but  not contained in $X$. Indeed, if $b_0\nsubseteq X$ is such a ball then there is some $1\leq i\leq s$ such that $b_i\sub b_0$ (for otherwise $b_0\sub b_i$ and hence $b_0\cap X=\0$). So there are only $s$-many possible such $0$-balls

Consequently, there exists some closed $0$-ball $b_0=B_{\geq 0}(x_0)\subseteq X$ with $b_0\subsetneq b$. If $s=0$ then $X=b$ is a ball of negative radius, since $b_0$ is closed. Thus, assume that $s\geq 1$, and  
for $1\leq i\leq s$ choose a center $x_i$ of $b_i$. Let $\gamma_1=\max\{v(x_0-x_i):1\leq i\leq s\}$. Since $b_0$ is disjoint from each of the $b_i$, necessarily $v(x_0-x_i)<0$ for all $i$ and hence $\gamma_1<0$. As a result, the open ball $B_{>\gamma_1}(x_0)$ properly contains $B_{\geq 0}(x_0)$ (since $\Gamma$ is dense) and is still disjoint from the $b_i$. In particular, we do not have $b\sub B_{>\gamma_1}(x_0)$, so we must have $B_{>\gamma_1}(x_0)\subsetneq b$.

It follows that $B_{>\gamma_1}(x_0)$ is contained in $X$, thus we found an open ball of negative radius in $X$. 
Finally, to get a closed such ball of negative radius, we use the fact that $\Gamma$ is dense.
\end{proof}

\subsubsection{The P-minimal case}
Recall the notion of P-minimal valued field from Section \ref{sss:p-minimal}. Let $\CK=(K,v,\dots)$ be a P-minimal valued field.

%By \cite[Theorem 2.2]{HasMac} any P-minimal field is p-adically closed.

Note that since $\Gamma$ is discrete, $\CB^{op}=\CB^{\cl}$.
It follows from Hensel's Lemma (see Fact \ref{HM} below) that each $P_n$ is open, implying that a definable subset $X\subseteq K$ is infinite if and only if it has non-empty interior. 

We thank D. Macpherson for sketching for us the proof the  proposition below. We first recall:

\begin{fact}\label{HM} \cite[Lemma 2.3]{HasMac}
Let $(K,v)$ be a $p$-adically closed field and let $n\in \mathbb N$ with $n>1$ and $x,y,a\in K$. Suppose that $v(y-x)>2v(n)+v(y-a)$. Then $x-a, y-a$ are in the same coset of $P_n$.
\end{fact}

\begin{proposition}\label{P:dugald}
Let $\CK=(K,v,\dots)$ be a P-minimal valued field. Then 

\begin{enumerate}
\item $\Gamma$ is definably complete and has definable Skolem functions.

\item $K$ satisfies {\minballn}.
\end{enumerate}

\end{proposition}
\begin{proof}

(1) Both properties follow from the fact that every definable set bounded below has a minimum.

(2)
Let $X\subseteq K$ be a definable subset intersecting infinitely many closed $0$-balls. We need to show that $X$ contains a closed ball of radius $-1$.

Partitioning $X$  into cells and translating by an element of $K$, we may assume that $X$ has the form
\[
\{x\in K: \gamma_1< v(x)<\gamma_2 \wedge P_n(\lambda\cdot x)\},
\]
where $\lambda\in K$, $\gamma_1<\gamma_2\in \Gamma\cup\{\infty,-\infty\}$ and $n\in\mathbb{N}$.
Since $\Gamma$ is discrete and $\bk$ is finite, the assumption that $X$ intersects infinitely many $0$-balls implies that $X$ is not contained in any ball $B_{\geq -m}(0)$  with $m\in \Nn$ (for every such ball intersects only finitely many closed $0$-balls). So for every $k\in \mathbb N$,  there is some $y_0\in X$  such that $v(y_0)<-k$.

We can thus fix   $y_0\in X$, such that $v(y_0)+2v(n)<-2$.
We claim that the closed ball $B_{\geq -1}(y_0)$ is contained in $ X$.

Indeed, assume that $v(x-y_0)\geq -1$. Then, since $v(y_0)<-2$, we have $v(x)=v(y_0)$ and therefore  $\gamma_1< v(x)=v(y_0)<\gamma_2$. Thus it is sufficient to see that $x$ and $y_0$ are in the same $P_n$-coset.

By our choice of $y_0$ we have $v(x-y_0)\geq -1>v(y_0)+2v(n)$, hence by Fact \ref{HM} (with $a=0$ there), $x$ and $y_0$ are in the same $P_n$-coset, so $x\in X$. Thus $B_{\geq -1}(y_0)\sub X$. We have thus shown that $X$ contains at least one closed ball of radius $-1$.
\end{proof}

\subsubsection{The weakly o-minimal case}

By \cite{CheDick}, the theory of real closed convexly valued fields,  known as \emph{Real Closed Valued Field},  RCVF for short, is weakly o-minimal.

% For sets $A,B$ in an ordered structure we let, as usual, $A<B$ if $x<y$ for all $x\in A$ and $y\in B$. Note that in a convexly valued field, if $b_1,b_2$ are balls then either one contains the other or $b_1<b_2$ or $b_2<b_1$. With this notation we have:

\begin{lemma} \label{general-rcvf} Let $(K,v)$ be a real closed valued field. 
If $X\sub K$ is a convex set and every open ball $b\sub X$ has nonnegative radius then $X$ is contained in a single closed ball of radius $0$.
\end{lemma}
\begin{proof}
We first claim that for each $x_1,x_2\in X$, we must have $v(x_1-x_2)\geq 0$. Indeed, assume that $v(x_1-x_2)=\gamma_0<0$ and let $x_0=\frac{x_1+x_2}{2}$. Then $$v(x_0-x_1)=v(x_0-x_2)=v((x_1-x_2)/2)=\gamma_0.$$ Thus the open ball
$B_{>\gamma_0}(x_0)$ does not contain $x_1,x_2$. Since the ball is convex and $X$ is convex we have $B_{>\gamma_0}(x_0)\sub X$, contradicting the assumption on $X$.
It follows that for every $x_0\in X$, the ball $B_{\geq 0}(x_0)$ contains $X$.
\end{proof}

%If $X$ contains no open ball of radius $0$ then the same proof shows that for every $x_1,x_2\in X$, we have $v(x_1-x_2)>0$. Thus for %any $x_0\in X$ we have $X\sub B_{>0}(x_0)$.  \qed

\begin{lemma}\label{properties wom}
Assume that $\CK=(K,<,v,\ldots)$ is a weakly o-minimal expansion of a real closed valued field such that $\Gamma$ and $\bk$ are o-minimal. Then
\begin{enumerate}
    \item $\Gamma$ is definably complete and has definable Skolem functions, and $\bk$ is surgical.
    \item $\CK$ satisfies {\minballn}.
\end{enumerate}
\end{lemma}
\begin{proof}
(1) Follows from the o-minimality of $\Gamma$ and of $\bk$.
(2) Let $X\sub K$ be a definable set intersecting infinitely many closed $0$-balls. By weak o-minimality, $X$ is a finite union of convex sets, so one of these intersects infinitely many closed $0$-balls.
By Lemma \ref{general-rcvf}, this component necessarily contains a ball of negative radius.\end{proof}

\begin{remark}\label{R:t-convex sat properties} 
By Section \ref{sss:t-convex}, T-convex power-bounded valued fields satisfy the assumptions of Lemma \ref{properties wom}.
\end{remark}

\subsection{Strong internality to the  distinguished sorts}
We can  finally show that, under suitable assumptions, any interpretable field $\CF$ is locally strongly internal to one of the distinguished sorts. We focus on the case where $\Gamma$ is dense as we do not know the results in the discrete case (and we do not need them for the proof of our main theorem). 

We first show that $K/\CO$ is an SW-uniform structure with respect to a natural topology: 
\begin{definition}
    Let $(K,v)$ be a valued field. The \emph{thick ball topology} on $K/\CO$ is the topology generated by $\{\pi(B): B\in \CB, r(B)<0\}$. Where $\pi: K\to K/\CO$ is the natural projection, and $r(B)$ is the valuative radius of  $B$. 
\end{definition}

\begin{lemma}\label{L:K/O is SW}
    Let $\CK=(K,v,\dots)$ be a dp-minimal valued field with a dense value group. If $K$ satisfies {\minballn} then $K/\CO$  is an SW-uniformity with respect to the thick ball topology. 
\end{lemma}
\begin{proof}
    By \cite[Example 3, page 3]{SimWal} to show that the thick ball topology  gives rise to a uniform structure on $K/\CO$ we only have check that it is a group topology. If we denote $\tau_a$ the collection of basic neighbourhoods of $a$ then $\tau_a=\tau_0+a$. Therefore the map $x\mapsto -x$ is a homeomorphism, and to check that addition is continuous it suffices to check continuity at $(0,0)$. This follows from the fact that the pre-image under addition (in $K^2$) of a ball $B\sub K$ containing $0$  is $B\times B$, and since $\pi(B\times B)=\pi(B)\times \pi(B)$ continuity of addition (in $K/\CO$) at $(0,0)$ follows. 
    
    So it remains to check that $K/\CO$ has no isolated points and that every infinite $S\sub K/\CO$ has non-empty interior. The first property follows from the fact that $\Gamma$ is dense (and therefore $\bk$ is infinite): if $a\in K/\CO$ and $a_0\in \pi^{-1}(a)$ is any point then a basic open neighbourhood of $a$ is the image under $\pi$ of a ball $B\ni a_0$ with radius $\gamma>0$. Since $\bk$ is infinite, $B$ contains infinitely many $0$-balls, and so $\pi(B)$ is infinite. The latter property is automatic from {\minballn}. 
\end{proof}

With this last observation in hand, we can finally show: 

\begin{proposition}\label{P:s-i to sorts, c-min and v-min}
Let $\mathcal{K}=(K,v,\dots)$ be either a C-minimal valued field or a weakly o-minimal convexely valued field whose value group and residue field are o-minimal. Any interpretable infinite field is locally strongly internal to either $K$, $K/\CO$, $\Gamma$ or $\bk$.
\end{proposition}
\begin{proof}
We may assume that $\CK$ is sufficiently saturated. Let $\CF=X/E$, $X\sub K^n$,  be an interpretable infinite field. By Proposition \ref{P:properties for Cmin} (for the C-minimal case) and Lemma \ref{properties wom} (for the weakly o-minimal case), the assumptions of Proposition \ref{new-main} hold.

Consequently there exists an infinite definable subset $S\subseteq \CF$ and a  finite-to-finite definable correspondence between $S$ and either $K$, $K/\CO$, $\Gamma$ or $\bk$.

To obtain strong internality from this finite-to-finite correspondence we aim to apply Lemma \ref{L:almost internal-to -strongly}. To do so, we need to show that the source of the correspondence, $\CF$,  satisfies {\efi} (which is clear) and the target (each of the distinguished sorts) is either an SW-uniformity, or, itself satisfies {\efi}. Since $K$ and $\bk$ are fields they both satisfy {\efi}, and so does  $\Gamma$ by virtue of being  linearly ordered.  This leaves the case of $K/\CO$. As $\CK$ satisfies {\minballn}, $K/\CO$ is an SW-uniformity by Lemma \ref{L:K/O is SW}.
\end{proof}

\section{Eliminating the sorts $K/\CO$ and $\Gamma$}

The results  collected up until this point allow us, given an infinite inpterpretable field $\CF$ to construct a finite-to-finite correspondence between a dp-minimal subset of $\CF$, strongly internal to $K/E$ (for some definable equivalence relation $E$) into one of $K$, $K/\CO$, $\Gamma$ or $\bk$. In the present section we develop the tools allowing  (under suitable assumptions) to eliminate $K/\CO$ and $\Gamma$ from the list.

\subsection{Opaque Equivalence Relations}
Towards studying  $K/\CO$,  we prove a general domination result for generic types in $(K/E)^n$ when $E$ is an \emph{opaque} equivalence relation. Recall;
\begin{definition}\cite[Definition 11.2]{HaHrMac2}
Let $D$ be a definable set in some structure and $E$ a definable equivalence relation on $D$.
\begin{enumerate}
    \item A definable set $X\subseteq D$ \emph{crosses} an $E$-class, $a/E$, if both $X$ and its complement in $D$ intersect $a/E$.
    \item The equivalence relation $E$ is $\emph{opaque}$ if every definable $X\subseteq D$ crosses at most finitely many classes.
\end{enumerate}
We will also refer to $D/E$ being opaque, meaning that $E$ is.
\end{definition}

\begin{example}
\begin{enumerate}
    \item If $(K,v,\dots)$ is $C$-minimal then $K/\CO$ and $K/\m$ are opaque, see \cite[Lemma 11.13(i)]{HaHrMac2}.
    \item If $(M,<,\dots)$ is weakly o-minimal then every definable convex equivalence relation on $K$ is opaque. Indeed, if $X$ consists of $r$ convex sets then since each $E$-class is convex, $X$ crosses at most $2r$-many classes.
    \item If $(K,v,\dots)$ is P-minimal then $K/\CO$ is opaque. This follows, essentially, from  Fact \ref{HM} (and quantifier elimination).
\end{enumerate}
\end{example}

The next lemma can be viewed as a statement on domination by opaque imaginary sorts.

Below, to simplify notation we use $\pi$ to denote each of the projections $D_i\to D_i/E_i$
and let $\pi_k:\prod_{i=1}^k D_i \to \prod_{i=1}^k(D_i/E_i)$ be the natural projection.
\begin{proposition}\label{P:opaque-dim}
Let $\CM$ be a structure of finite dp-rank.
Let $D_1,\ldots, D_n$ be any definable sets, and for each $i=1,\ldots, n$, let $E_i$ be a definable opaque relation on $D_i$, such that the structure on $D_i/E_i$ is dp-minimal. Let $E=\prod_i E_i$ be the equivalence relation on $D= \prod_i D_i$.

Assume that $X\sub P\sub D$ are definable sets such that $\pi_n(X)=\pi_n(P)$ (i.e. $X$ intersects every $E$-class which  $P$ does).  Then $\dpr(\pi_n(P\setminus X))<n$.
\end{proposition}
\begin{proof}
We proceed by induction on $n$. For $n=1$, we may assume that $P=D_1$ and  $X\subseteq D_1$ is definable and intersects every $E$-class.
It follows that $D_1\setminus X$ intersects only those classes which $X$ crosses, and by opacity  there are only finitely many such $E$-classes.
Thus $\pi_1(D_1\setminus X)$ is finite so has $\dpr(\pi_1(D_1\setminus X))=0$, as required.

Assume that the proposition is proved for all $n'< n$, where $n>1$, and now let $P$ and $X$ be as in the statement. We aim to prove that $\dpr(\pi_n(P\setminus X))<n$.

We begin with some notation and elementary observations.
We let $D'=\prod_{i=1}^{n-1} D_i$. For $i=1,\ldots,n$ we write $\overline{D}_i=D_i/E_i$ and $\overline{D}=D/E$.

For $a\in D'$, $b\in D_n$ and $Y\sub D$, we let
 \[Y_a=\{b'\in D_n:(a,b') \in Y\}\,\, ,\,\, Y^b=\{a'\in D':(a',b) \in Y\}.\]

 We shall sometimes identify elements in $\overline{D}$ with the corresponding $E$-class inside $D$. In particular, an element $\beta\in \overline{D}_n$ is also a subset of $D_n$, and for $Y\sub D$ we define
    \[Y^\beta:=\bigcup_{b\in \beta}Y^b=\{a\in D':\exists b\in \beta\, (a,b)\in Y\}.\]
    It is a subset of $D'$ and we have $\pi_{n-1}(Y^\beta)=(\pi_n(Y))^\beta$ (where on the right, $\beta$ is taken as an element of $\overline{D}_n$), and for $a\in D'$ and $\beta\in D_n$, we have $a\in Y^\beta\Leftrightarrow \beta\cap X_a\neq\0$.

Since $\pi_n(P\setminus X)\sub \pi_n(P)$, we may assume that $\dpr(\pi_n(P))=n$. Let \[X^*=\{(a,b)\in (D'\times D_n)\cap X:  X_a \mbox{ crosses }  b/E_n \}.\]
Clearly $X^*\sub X$ is definable.
\vspace{.1cm}

\noindent ($\clubsuit$)
 \textbf {A special case}: $X^*=\emptyset$ (i.e for all $a\in D'$, $X_a$ does not cross any $E_n$-class).

Since $\pi_n(X)=\pi_n(P)$ it immediately follows that for every $\beta \in \oD_n$ we get $\pi_{n-1}(X^\beta)=\pi_{n-1}(P^\beta)$.
For each $\beta\in \oD_n$ we may apply induction to $X^\beta\sub P^\beta\sub D'$ to conclude that $\dpr(\pi_{n-1}(P^\beta\setminus X^\beta))<n-1$. 

\begin{claim}
 $P^\beta\setminus X^\beta=(P\setminus X)^\beta$.
\end{claim}
 \begin{claimproof}
 The inclusion $\sub$ is true without any assumptions on $X$. For the converse, assume that $a\in (P\setminus X)^\beta$, namely there exists $b\in \beta$ such that $(a,b)\in P\setminus X$. It follows that $a\in P^\beta$ so we want to show that $a\notin X^\beta$. Indeed, since $a\notin X^b$ then $b\notin X_a$. But since $X_a$ does not cross any $E_n$-class then $\beta\cap X_a=\0$, and therefore $a\notin X^\beta$, so $a\in P^\beta\setminus X^\beta$.
 \end{claimproof}

Using the claim and the induction hypothesis we conclude that for every $\beta\in \oD_n$,
\[\dpr(\pi_n(P\setminus X)^\beta)=\dpr\left( \pi_{n-1}((P\setminus X)^\beta)\right)=\dpr\left( \pi_{n-1}(P^\beta\setminus X^\beta)\right)<n-1.\]

Since this is true for every $\beta\in \oD_n$, it follows from the sub-additivity of rank that $\dpr(\pi_n(P\setminus X)<n.$ This concludes the proof of the special case. \qed$_\clubsuit$ 

We return to the general case and first note that $\dpr(\pi_n(X^*))<n$. Indeed, let $\widehat{X^*}=\{( a,\pi(b))\in D'\times \overline{D_n}: (a,b)\in X^*\}$. By the definition of $X^*$ and opacity, the projection of $\widehat{X^*}$ to $D'$ is finite-to-one and hence $\dpr(\widehat{X^*})<n$. It follows that $\dpr(\pi_n(X^*))<n$ as well.

To finish the proof in the general case, let $R= \pi_n(X^*)$,   $P_1=P\setminus \pi_n^{-1}(R)$
and $X_1=X\cap P_1$. Now, $X_1^*=\0$ (because $X_1^*\sub X^*$ and $P_1\cap X^*=\0$) and $\pi_n(X_1)=\pi_n(P_1)$ so by the special case $\clubsuit$,
we have $\dpr(\pi_n(P_1\setminus X_1))<n$. Finally, since $P\setminus X\subseteq (P_1\setminus X_1)\cup\pi_n^{-1}(R)$
\[\pi_n(P\setminus X)\sub \pi_n(P_1\setminus X_1)\cup R.\] We have seen in the previous paragraph that $\dpr(R)<n$ so $\dpr(P\setminus X)<n$.
\end{proof}

We now prove the domination result referred to above:
\begin{lemma}\label{L:dominated by generic}
    Let $\CM$ be a dp-minimal structure and $E$ an $A$-definable opaque equivalence relation on $M$. For any  complete type $p\in S(A)$ concentrated on $\CM^n/E^n$ with $\dpr(p)=n$, there exists a unique complete type $q$ over $A$  , concentrated on $M^n$ with $\pi_*q=p$. Furthermore, $\dpr(q)=n$ as well.
\end{lemma}
\begin{proof}
For simplicity, we denote $E^n$ by $E$ and let $\pi:M^n\to M^n/E$ be the quotient map. Note that since $\pi$ is surjective, there exists at least one  complete type, $q$ over $A$ such that $\pi_*(q)=p$.

    To show uniqueness assume towards a contradiction that there exists an $A$-definable subset $Z\subseteq M^n$ such that both $p\vdash \pi(Z)$ and $p\vdash \pi(Z^c)$ (where $Z^c:=M^n\setminus Z$). Then $p\vdash \pi(Z)\cap \pi(Z^c)$, hence $\dpr(\pi(Z)\cap \pi(Z^c))=n$. Let $P:=\pi^{-1}(\pi(Z)\cap \pi(Z^c))$ and $X=Z\cap P$. By definition of $P$, $\pi(Z\cap P)=\pi(P)=\pi(Z^c\cap P)$. But since $P\setminus X=P\cap Z^c$, $\pi(P\setminus X)=\pi(P)=\pi(X)$. By Proposition \ref{P:opaque-dim}, $\dpr(\pi(P\setminus X))<n$ and $\dpr(\pi(X))<n$, so $\dpr(\pi(P))<n$ contradicting the above.
\end{proof}
\begin{remark}\label{R:domination}
In the notation of the previous lemma, it  is not hard to see that $q=\{Z\subseteq M^n: \text{$Z$ is $A$-definable, } p\vdash \pi(Z)\}$.
\end{remark}

\subsection{Definable functions in $K/\CO$.}\label{Sec:K/O}

Our goal in the present section is to show that, under certain assumptions, definable functions on $(K/\CO)^n$ are locally affine (with respect to the group structure on $K/\CO$). This fact, of possible interest on its own right, will allow us to show that, under the same assumptions, no infinite field interpretable in $K$ is locally strongly internal to $K/\CO$. Note that the analogous statement fails if one replaces $K/\CO$ with  $K/\m$, as $\bk$ itself is strongly internal to $K/\m$.

\begin{assumption}
Until the end of this section we let $\CK=(K,v,\dots)$ be a sufficiently saturated dp-minimal valued field of characteristic $0$ satisfying {\minballn}, such that $\Gamma$ is dense (equivalently $\bk$ is infinite). Throughout the topology on $K/\CO$ is the thick ball topology and the topology on $(K/\CO)^n$ is the product topology. 
\end{assumption}

%\begin{remark}
% Recall that {\minball} is implied by {\ominball}.
%\end{remark}

%We fix the associated topology on $K/\CO$ and on $(K/\CO)^n$.

Let $\pi:K\to K/\CO$ be the quotient map. For simplicity, for every $n$, we still denote by $\pi$ the projection $K^n\to (K/\CO)^n$.

\begin{definition}
For $A$ a set of parameters (from any of the sorts),  we say that a partial type $P$ over $A$ concentrated on $K^n$ is \emph{thick} if $\dpr(\pi_*P)=n$.
\end{definition}

\begin{example}\label{E: thick ball}
An open ball $B\subseteq K$ is thick if and only if $\pi(B)$ has non-empty interior in $K/\CO$ if and only  if $r(B)<0$.
\end{example}

\begin{lemma}\label{L:open in intersection}
Let $\CM$ be any (multi-sorted) structure and let $D$ be an SW-uniformity. Let $\CM'\succ \CM$ a sufficiently saturated elementary extension. If $p\in S(A)$ is a complete type concentrated on $D^n$ with $\dpr(p)=n$ then for any $a\models p$, there exists an $\CM'$-definable open set $X$, such that  $a\in X\subseteq p(\CM')$.
\end{lemma}
\begin{proof}
Let $\{\theta(x,y,t):t\in T\}$ be the definable family given by the definition of a definable uniformity an let $D'=D(\CM')$. For every $A$-formula $\phi\in p$, let $\theta(D',a,t_\phi)$ be a non-empty open subset of $\phi(\CM')$ containing $a$. Thus $\bigcap_\phi \theta(D',a,t_\phi)$ contains $a$ and is a subset of $p(\CM')$. By saturation, we can find $t_0\in D'$ with $a\in \theta(D',a,t_0)\subseteq p(\CM')$.
\end{proof}

\begin{lemma}\label{L:basic properties of thick}
Let $(K,v,\dots)$ be as above and let $A$ be an arbitrary set of parameters.
    \begin{enumerate}
        \item A partial type $P\vdash K^n$ over $A$ is thick if and only if there is a  completion $p$ of $P$ (over $A$) which is thick.
        \item If a partial type $P\vdash K^n$ is thick then $\dpr(P)=n$.
        \item If $\tp(a_1,\dots,a_n/A)$ is thick then so is $\tp(a_1/Aa_2,\dots,a_n)$.
        \item Let $B\subseteq K^n$ be a thick open polydisc. Then $B+\CO^n= B$.
        \item For any thick complete type $p$ and $a\models p$, there exists a thick open polydisc $X$ (possibly defined over additional parameters) satisfying $a\in X\subseteq p(K)$.
    \end{enumerate}
\end{lemma}
\begin{proof}
(1) Assume that $P$ is concentrated on $K^n$. The right-to-left direction is by definition. For the other direction, let $q$ be a completion of $\pi_*P$ with $\dpr(q)=n$. Then, by Lemma \ref{L:dominated by generic}, there is a unique (thick) type $p$ with $\pi_*p=q$. Consequently, $p$ must be a completion of $P$.

(2) Let $p$ be any thick completion of $P$, as supplied by (1). By Lemma \ref{L:dominated by generic}, $\dpr(p)=n$ and so $\dpr(P)=n$ as well.

(3) For ease of writing assume that $n=2$ (the proof is the same for larger $n$). As \[\dpr(\pi(a_1),\pi(a_2)/A)=2\] and \[2\geq \dpr(\pi(a_1),a_2/A)\geq \dpr(\pi(a_1),\pi(a_2)/A),\] we have that that $\dpr(\pi(a_1),a_2/A)=2$. By sub-additivity, \[2=\dpr(\pi(a_1),a_2/A)\leq \dpr(\pi(a_1)/Aa_2)+\dpr(a_2/A)\leq 2,\] so $\dpr(\pi(a_1)/Aa_2)=1$. As needed.

(4) It is enough to consider the case when $n=1$. Let $a\in B$ and $c\in \CO$. As $B$ is thick, $B=B_{>\gamma}(a)$ for some $\gamma<0$ (see Example \ref{E: thick ball}). Then $v(a+c-a)=v(c)\geq 0>\gamma$ and thus $a+c\in B$.

(5) Since $K/\CO$ is an SW-uniformity, by Lemma \ref{L:open in intersection} we can find a definable open subset $U$ of $(\pi_*p)(K)$ containing $\pi(a)$. By shrinking $U$, we may assume that it is a product of basic open sets, i.e. a product of images (under $\pi$) of balls of radius greater than $0$). Thus $\pi^{-1}(U)$ is a thick open polydisc containing $a$. By the remark following Lemma \ref{L:dominated by generic}, $\pi^{-1}(U)\subseteq p(K)$.
\end{proof}

For the rest of the section we need  definable functions in our valued fields to satisfy some additional geometric assumptions. These assumptions arise naturally in the context of 1-h-minimal valued fields discussed in Section \ref{ss:1-h-minimality and its connections} below.

\begin{vjp}[Valuative Jacobian Property]
For any set of parameters, $A$, and  $A$-definable function $f:K\to K$, there exists a finite $A$-definable set $C\sub K$ such that for every ball $B$ disjoint from $C$, the derivative $f'$ exists on $B$, $v(f')$ is constant on $B$ and moreover:

\begin{enumerate}
		\item For every $x_1,x_2\in B$
		\[
		v(f(x_1)-f(x_2))=v(f'(x_1))+v(x_1-x_2).
		\] (written multiplicatively, $|f(x_1)-f(x_2)|=|f'(x_1)||x_1-x_2|$.)
		\item If $f'\neq 0$ on $B$ then for every open ball $B'\sub B$ the image $f(B')$ is an open ball of radius $v(f')+r(B')$ where $r(B)$ is the valuative radius of $B$.
	\end{enumerate}
\end{vjp}

\begin{notation}
    For any definable $K$-differentiable (partial) function $f:K^n\to K$  let $f_{x_i}:=\frac{\partial f}{\partial x_i}$ and $\nabla f=(f_{x_1},\dots, f_{x_n})$  the gradient of $f$.
\end{notation}

\begin{mvtay}[Multivariate Valuative Version of Taylor's Approximation]
	Given any set of parameters $A$ and an $A$-definable function $f:K^n\to K$, there exists an $A$-definable  set $C\subseteq K^n$ with empty interior such that for any polydisc $B\sub K^n\setminus C$ $f$ is $2$-times differentiable on $B$ and
	\[
	v(f(x)-f(x_0)-\nabla f(x_0)(x-x_0))\ge \min_{1\le i,j\le n}\{ v(f_{x_{i},x_{j}}(x_0))+v((x-x_0)^{(i,j)})\},
	\]Where $(c_1,\dots,c_n)^{(i,j)}=c_{i}c_{j}$. 
	
	(Written multiplicatively,
	\[|f(x)-f(x_0)-\nabla f(x_0)(x-x_0)|\leq  \max_{1\le i,j\le n}\{ |f_{x_{i},x_{j}}(x_0)| \cdot |(x-x_0)^{(i,j)}|\}.)\]

\end{mvtay}

By a standard induction and using {\VJP} (see \cite[Theorem 5.6.1]{hensel-min})) it is not hard to show that {\MVTay} follows from its one dimensional version. Although we will not require both {\VJP} and {\MVTay} for each of the following results, it is cleaner to assume both.

\begin{remark}\label{R:finite-set-around thick type}
Let $p$ be a thick complete type over $A$ concentrated on $K$. Since it is necessarily not algebraic, $p(K)\cap C=\emptyset$ for any finite $A$-definable set $C$. Furthermore,  by Lemma \ref{L:basic properties of thick}(5), for every $a\models p$ there exists a thick open ball $B$, $a\in B\subseteq p(K)$ and thus $B\cap C=\emptyset$.
\end{remark}

We make use of the following assumption.
\begin{assumption}[$\spadesuit$]
$(K,v,\dots)$ is a sufficiently saturated dp-minimal valued field of characteristic $0$ satisfying:
\begin{itemize}
    \item {\minballn}
    \item $\Gamma$ is dense
    \item {\gendif}
    \item {\VJP}
    \item {\MVTay}
\end{itemize}
\end{assumption}

\begin{definition}
A (partial) function $f:K^n\to K$ \emph{descends} to $K/\CO$ if  $\dom(f)+\CO^n=\dom(f)$ and for every $a,b\in \dom(f)$, if  $a-b\in \CO^n$, then $f(a)-f(b)\in \CO$. The function  $f$ {\em descends to $K/\CO$ on some (partial) type $P\vdash \dom(f)$} if $f\restriction P$ descends to $K/\CO$.
\end{definition}

\begin{example}\label{example-descent}
When $a\in \CO$  the linear function $\lambda_a:x\mapsto a\cdot x$ descends to an endomorphism $\tilde \lambda_a:(K/\CO,+)\to (K/\CO,+)$. If $a\in\m$, then $\tilde \lambda_a$  has an infinite kernel. Thus we obtain a definable locally constant, surjective endomorphism of $K/\CO$.
\end{example}

\begin{lemma}\label{L:derivtives in O}
	Assume that $(K,v,\dots)$ satisfies Assumption $\spadesuit$. Let $p$ be a complete thick type in $K^n$, $p\vdash \dom(f)$ for some definable partial function $f: K^n\to K$. Then
	\begin{enumerate}
	    \item $f$ is differentiable on $p$;
	    \item if $f$ descends to $K/\CO$ (on $p$) then $\nabla f(a)\in \CO^n$ for all $a\models p$;
	    \item  Assume that  $\mathrm{Im} (f)\subseteq \CO$. Then for every $a\models p$  there exists a thick open polydisc $B$, $a\in B\sub p(K)$, such that for all $b\in B$ and $1\leq i\leq n$
	\[
	v(f_{x_i}(b))+2r(B) > 0,
	\] where, for a polydisc $B:=\prod B_i$ we denote $r(B):=\max\limits_{i=1,\dots, n} r(B_i)$.
	In particular, $f_{x_i}(a)\in \m$ for all $a\models p$.
	\end{enumerate}
\end{lemma}
\begin{proof}
(1) By Lemma \ref{L:dominated by generic}, $\dpr(p)=n$. The result follows by {\gendif}.
	%	By \cite[Theorem 5.1.5]{hensel-min} $f$ is differentiable on a dense subset of $K^n$. Since $K$ is dp-minimal it is an SW-uniformity and since $p$ is thick we have $\dpr(p)=n$, so $p\vdash \neg S$ for any definable set $S$ with empty interior. So $f$ is differentiable on $p$.

(2) 	
	Let $a:=(a_1,\dots, a_n)\models p$. Without loss of generality, we show that $f_{x_1}(a)\in \CO$. By Lemma \ref{L:basic properties of thick}(3), $p_1:=\tp(a_1/a_2,\dots, a_n)$ is a thick type.
	
	Applying {\VJP} to $g(t):=f(t,a_2,\dots, a_n)$ we obtain an $(a_2,\dots, a_n)$-definable finite set $C$ such that $g'$ is constant on any ball disjoint from $C$. By Remark \ref{R:finite-set-around thick type}, $C\cap p_1(K)=\0$ and by Lemma \ref{L:basic properties of thick}(5), there is a thick open ball $B_1$ such that $a_1\in B_1\sub p_1(K)$. By Lemma \ref{L:basic properties of thick}(4), $a_1+1\in B_1$. Therefore, by {\VJP}(1),
	\[
v(g(a_1+1)-g(a_1))=v(g'(a_1))+v(1)=v(g'(a_1)) \text{ and so}\]
	\[
	v(f(a_1+1,\dots, a_n)-f(a_1,\dots, a_n))=v(f_{x_1}(a)) .
	\]
	Since $f$ descends to $K/\CO$ we get that $v(f(a_1,\dots, a_n)-f(a_1+1,\dots, a_n))\geq 0$, and so $v(f_{x_1}(a))\geq 0$.
	
(3) 
Let $B\subseteq p(K)$ be a thick polydisc containing $a$, as provided by Lemma \ref{L:basic properties of thick}(5).  Assume that $B=B_1\times\dots\times B_n$. We assume that $\mathrm{Im}(f)\sub \CO$.
\begin{claim}
For any $1\leq i\leq n$ and $b\in B$, $v(f_{x_i}(b))+r(B_i)\geq 0$. In particular, $v(f_{x_i}(b))+r(B)\geq 0$.
\end{claim}
\begin{claimproof}
Let $\hat b=\{b_1,\dots,b_{i-1},b_{i+1},\dots,b_b\}$ and let $C$ be the $\widehat b$-definable finite set provided by {\VJP} with respect to the $\widehat b$-definable function $g(x_i)=f(b_1,\dots,x_i,\dots b_n)$. As $B_i\subseteq \tp (b_i/A\widehat b)$ and the latter is thick by Lemma \ref{L:basic properties of thick}(3) (so non-algebraic), $B_i\cap C=\emptyset$, hence $v(g')$ is constant on $V_i$. If $g'(x_i)\equiv 0$ on $B_i$ then the inequality holds trivially. Otherwise by {\VJP}(2), $g(B_i)$ is an open ball of radius $v(g'(b_i))+r(B_i)$. As, by assumption, $g(B_i)\subseteq \CO$, the result follows.
\end{claimproof}
Clearly, the claim holds with $B$ replaced by any $B'\subseteq B$.
For every $B'=\prod_i B_i'\subseteq B$ a thick open polydisc, $r((B')_i)<0$ for all $i$.  As $\Gamma$ is dense, we can find a (thick) open polydisc $B'$, $a\in B'\subseteq B$, with $0>2r(B')>r(B)$.

Now, for every $b\in B'$, we have $v(f_{x_i}(b))+2r(B')>v(f_{x_i}(b))+r(B)\geq 0$,  completing the proof.\end{proof}

We can now prove:
\begin{lemma}\label{L:near-affine in K/O}
	Assume that $(K,v,\dots)$ satisfies Assumption $\spadesuit$. Let $f:K^n\to K$ be an $A$-definable partial function and $p\vdash \dom(f)$ a complete thick type over $A$. If $f$ descends to $K/\CO$ (on $p$) then for every $a\models p$ there is a thick polydisc $B$, $a\in B\sub p(K)$, such that for all $x\in B$.
	\[
	f(x)-f(a)-\nabla f(a)(x-a)\in \m.
	\]

\end{lemma}
\begin{proof}
    By Lemma \ref{L:derivtives in O}(2), $\nabla f(c)\in \CO^n$ for all $c\models p$. We may thus assume that $\nabla f(c)\in \CO^n$ for all $c\in \dom(f)$.

	For every $A$-definable set $C\subseteq K^n$ with empty interior, $p(K)\cap C=\emptyset$ (since, by Lemma \ref{L:dominated by generic}, $\dpr(p)=n$). Let $a\models p$ and $B_0\subseteq p(K)$ be a thick polydisc containing $a$, as given by Lemma \ref{L:basic properties of thick}(5). By {\MVTay}, for every $x\in B_0$,
	
	\begin{equation}\label{eqn:3}
	v(f(x)-f(a)-\nabla f(a)(x-a))\ge \min_{1\le i,j\le n}\{ v(f_{x_{i},x_{j}}(a))+v((x-a)^{(i,j)})\}
	\end{equation}
	
	For any $1\leq i\leq n$, $\mathrm{Im}(f_{x_{i}})\subseteq \CO$. Thus Lemma \ref{L:derivtives in O}(3), applied to $f_{x_i}$,  gives a thick open polydisc $B^i$, $a\in B^i\subseteq p(K)$, such that for all $1\leq j\leq n$ and $b\in B^i$,
	$$v(f_{x_i,x_j}(b))+2r(B^i)>0.$$
	
	As passing to an open sub-polydisc does not affect the above and a finite (non-empty) intersection of thick open poyldisc is still such, we may replace the $B^i$ by  $B=\bigcap_{0\leq i\leq n} B^i$. We write $B=\prod_i B_i$ and by reducing $B$ further, we may assume that $r(B)=r(B_i)$ for all $1\leq i\leq n$.
	
	By our choice of $B$, for every $(i,j)$, we have $v(f_{x_{i}x_{j}}(a))+2r(B)> 0$. Also, for every $x\in B$ $v((x-a)^{(i,j)}>2r(B)$. Thus, it follows from Equation \eqref{eqn:3} that  $v(f(x)-f(a)-\nabla f(a)(x-a))>0$, as required. 
\end{proof}

Putting together the results proved thus far we can now show that, under our standing assumptions,  definable functions on $(K/\CO)^n$  lifting to $K^n$ are locally affine.

\begin{proposition}\label{P:locally affine}
Let $\CK=(K,v\dots)$ be a sufficiently saturated dp-minimal valued field of characteristic $0$ satisfying Assumption $\spadesuit$, i.e.:
\begin{itemize}
    \item {\minballn}
    \item $\Gamma$ is dense
    \item {\gendif}
    \item {\VJP}
    \item {\MVTay}
\end{itemize}
Let $f:(K/\CO)^n\to K/\CO$ be an $A$-definable partial function  with $\dom(f)$ open. If $f$ lifts to $K^n$, i.e there exists a definable partial function $\widehat f:K^n\to K$  descending to $f$, then there exists an open definable  $U\subseteq \dom(f)$, a definable homomorphism   $L:(K/\CO,+)^n\to (K/\CO,+)$ and $d\in K/\CO$ such that for every $y\in U$,
\[f(y)=L(y)+d.\]
\end{proposition}
\begin{proof}
Let $\widehat f:K^n\to K$ be a lift of $f$, namely for every $x\in \dom(\widehat f)$, $\widehat f(x)$ is in the $\CO$-coset $f(\pi(x))$. For simplicity assume that $\widehat f$ is also definable over $A$.

Let $c\in \dom(f)$ be with $\dpr(c/A)=n.$ By Lemma \ref{L:dominated by generic}, there exists a unique complete type $p$ concentrated on $(K/\CO)^n$ with $\pi_*(p)=\tp(c/A)$ and $\dpr(p)=n$. Clearly, $\widehat f$ descends to $K/\CO$ on $p$.

Let $\alpha\models p$. By Proposition \ref{L:near-affine in K/O}, there is a thick poyldisc $B$, $\alpha\in B\subseteq p(K)$ such that for all $x\in B$
\[\widehat f(x)-\widehat f(\alpha)-(\nabla \widehat f(\alpha))(x-\alpha)\in \m.\]

By  Lemma \ref{L:derivtives in O}(2), $\nabla \widehat f(\alpha)=(a_1,\ldots, a_n)\in \CO^n$. As noted in Example \ref{example-descent}, each linear map $x\mapsto a_i x$ descends to an endomorphism of $K/\CO\to K/\CO$  and thus $\nabla f(\alpha)$ descends to a group homomorphism  $L:(K/\CO)^n\to K/\CO$.
 Notice that $U=\pi(B)$ is an open subset of $(K/\CO)^n$. Since $\pi(\m)=0$,  for all $y\in U$,
$f(y)=L(y)+(f(\pi(\alpha))-L(\pi(\alpha))).$
\end{proof}

Combined with Corollary \ref{C:not-internal-to locally affine} (and Lemma \ref{L:K/O is SW}), we obtain:
\begin{corollary}\label{eliminating K/O}
Let $(K,v,\dots)$ be a valued field satisfying  Assumptions $\spadesuit$  and with the additional property that every definable partial function from $(K/\CO)^n$ into $K/\CO$ can be lifted to a definable (partial) function on $K^n$. If $\CF$ is an interpretable field in $K$ then $\CF$ is not locally strongly internal to $K/\CO$.
\end{corollary}

It remains to investigate the assumption that definable functions on $(K/\CO)^n$ can be lifted (definably) to  $K^n$. Consider the following assumption on a valued field $(K,v)$:
\begin{bcen}
For $A\sub K$,  every $A$-definable closed ball has an $A$-definable element.
\end{bcen}

\begin{remark}\label{R: sufficent for star}
\begin{enumerate}
    \item If $\dcl(A)$ is an elementary substructure for any such $A$, i.e. if $\CK$ has definable Skolem functions, then {\BCen} holds.
    \item If $K$ has residue characteristic $0$, then it is enough for $\acl(A)$ to be an elementary substructure. Indeed, as noted in \cite{HrKa},  every ball is closed under (finite) averages of elements, so whenever a ball contains a finite $A$-definable set it also contains a point in $\dcl(A)$.
\end{enumerate}
\end{remark}

\begin{lemma}\label{L: Lifts to K}
    Let $(K,v\dots)$ be a valued field satisfying {\BCen} and let $A\sub K$. Then every $A$-definable partial function $f:(K/\CO)^n\to K/\CO$ lifts to an $A$-definable partial function $K^n\to K$.
\end{lemma}
\begin{proof}
    As a first stage we lift $f$ to an $A$-definable function $f_0:K^n\to K/\CO$, by setting $f_0=f\circ \pi$. By {\BCen}, for every $a\in K^n$ there is an $aA$-definable element $F_a\in f_0(a)$. %Therefore there   exists an $A$-definable partial function 
    The function $\widehat f:K^n\to K$ defined by $\widehat f(a)=F_a$ satisfied the requirements. 
\end{proof}

\subsection{1-h-minimality, V-minimality and $T$-convex structures}\label{ss:1-h-minimality and its connections}
Our goal in this section is to describe various model theoretic settings, mainly in equi-characteristic $0$,  in which the assumptions {\gendif}, {\VJP} and {\MVTay}, from Section \ref{Sec:K/O} are satisfied. As we shall see,  these assumptions hold for  V-minimal and power bounded $T$-convex structures, two settings in which our main result, Theorem \ref{T: main interp}, on interpretable fields, is proved.

A natural context in which these assumptions hold is that of  Hensel-minimal valued fields, introduced by Cluckers, Halupzcok and Rideau-Kikuchi (in equi-characteristic $0$) in \cite{hensel-min}.  As is shown in
\cite[\S 6.3]{hensel-min}, power-bounded $T$-convex theories are $1$-h-minimal, and by \cite[Proposition 6.4.2]{hensel-min} so are the V-minimal fields of Hrushovski-Kazhdan \cite{HrKa}.
We give a brief overview of what we need in order to introduce these various notions.

The following definition is equivalent to $1$-h-minimality in equicharacteristic $0$  (\cite[Theorem 2.9.1]{hensel-min}):

\begin{definition}\label{D:1hmin}
	Let   $T$ be a theory extending that of valued fields of equi-characteristic $0$. $T$ is called {\em $1$-h-minimal} if for every $\CK=(K,v,\dots)\models T$ and  every parameter set $A\sub K\cup \mathrm{RV}$ and an $A$-definable $f: K\to K$ :
	\begin{enumerate}
		\item There exists a finite $A$-definable set $C$ such that for any ball $B$ disjoint from $C$ there exists $\mu_B\in \Gamma$ such that $v(f(x)-f(x'))=\mu_B+v(x-x')$ for all $x,x'\in B$. Written multiplicatively: $|f(x)-f(x')|=\mu_B|x-x'|.$
		\item The set $\{d\in K: f^{-1}(d) \text{ is infinite}\}$ is finite.
	\end{enumerate}
\end{definition}

If we assume that $\CK\models T$ is, additionally, dp-minimal (so, as noted in Example \ref{E:SW-uniformities},  an SW-uniformity) then condition (2) above is equivalent to having no definable locally constant functions with infinite image. By \cite[Proposition 5.2]{SimWal}, this is equivalent in this setting to $\acl(\cdot)$ satisfying the exchange principle.

Collecting the results from \cite{hensel-min} we conclude:
\begin{fact}\label{F:1hmin}
	Let $T$ be a $1$-h-minimal theory.  Then every model of $T$ satisfies {\gendif}, {\VJP} and {\MVTay}.
\end{fact}
%In fact, \cite[Theorem 5.6.1]{hensel-min}, gives a valuative Taylor approximation of any order, from which all the above properties follow.

\begin{proof} All references are  to \cite{hensel-min}:   {\gendif} is Theorem 5.1.5, {\MVTay} is Theorem 5.6.1, and {\VJP} is Corollary 3.1.6\end{proof}

Recall the definition of V-minimal valued fields from Section \ref{sss:v-minimal}.

\begin{proposition}\label{V-minimal} 
Every V-minimal theory is $1$-h-minimal, and in addition satisfies {\minballn} and  {\BCen}. \end{proposition}
\begin{proof} The fact that V-minimal theories are $1$-h-minimal is \cite[Proposition 6.4.2]{hensel-min}.

By Proposition \ref{P:properties for Cmin}, every V-minimal theory satisfies {\minballn}.
As for {\BCen}, notice that property (3) in its definition implies that every $A$-definable closed ball $b$ contains an $A$-definable finite set. If the residue characteristic is $0$ then the average of this finite set gives an $A$-definable point in $b$.
\end{proof} 

To summarize,  every V-minimal theory satisfies all the assumptions of Proposition \ref{P:locally affine} and in addition, by Remark \ref{R: sufficent for star}, every definable function on $(K/\CO)^n$ lifts to a definable function on $K^n$.

\vspace{.3cm}

Recall the definition of T-convex power-bounded valued fields from Section \ref{sss:t-convex}.

\begin{proposition}\label{L:Tconv to 1h}
	A power bounded $T$-convex theory $T_{conv}$ is $1$-h-minimal, and in addition satisfies {\minballn}. If we add to the language a constant outside of $\CO$ then it also satisfies {\BCen}.
\end{proposition}
\begin{proof} The fact that $T_{conv}$ is $1$-h-minimal is \cite[Theorem 6.3.4]{hensel-min}),  based on the work of Yin \cite{Yin}.
By Remark \ref{R:t-convex sat properties}, it satisfies {\minballn}. By v.d.Dries \cite[Remark 2.4]{vdDries-Tconvex}, if we add a constant for an element outside of $\CO$ then the theory has definable Skolem functions, so in particular satisfies {\BCen}.\end{proof}

So,
every power bounded $T$-convex theory satisfies all the assumptions of Proposition \ref{P:locally affine} and in addition, every definable function on $(K/\CO)^n$ lifts to a definable function on $K^n$.

\subsection{Eliminating $K/\CO$}\label{S:K/O}

Putting the observations from Section \ref{ss:1-h-minimality and its connections} together with Proposition \ref{P:locally affine} and Proposition  \ref{C:not-internal-to locally affine}
we obtain:
\begin{proposition}\label{V+convex:Eliminating K/O}
Assume that $\CK=(K,v\dots)$ is either V-minimal or power bounded $T$-convex and that $\CF$ is an infinite interpretable field. Then $\CF$ is not locally strongly internal to $K/\CO$.
\end{proposition}

\begin{remark} \begin{enumerate}
 \item The above shows that no infinite field is definable in the induced structure on $K/\CO$, however a field isomorphic to $\bk$ is {\em interpretable}  in this structure. Indeed, if $\gamma\in \Gamma$ is negative then the quotient $B_{\geq \gamma}(0)/B_{>\gamma}(0)$ is in definable bijection with $\bk$, and is easily seen to be interpretable in $K/\CO$.
    \item In an earlier version of this article \cite{HaPetRCVF}, we proved directly (not using the machinery of 1-h-minimality) that in any $T$-convex valued field (not necessarily power-bounded)  no interpretable field is locally strongly internal to $K/\CO$.
\end{enumerate}
\end{remark}
\vspace{.2cm}

In the P-minimal setting, it is easier to eliminate $K/\CO$. We first prove some general results.

\begin{lemma}\label{L:infinite ball in discrete}
Assume that $\CK=(K,v,\dots)$ satisfies {\minballn} and assume that $\bk$ is finite (and hence $\Gamma$ is discrete). If $X\subseteq K$ is an infinite definable set containing infinitely many closed $0$-balls then for every $k\in \mathbb{N}$, $X$ contains a ball of radius $-k$.

In particular, if $\CK$ is $\aleph_0$-saturated then $X$ contains a ball of radius $\gamma$ satisfying $\gamma<-k$ for all $k\in \mathbb{N}$.
\end{lemma}
\begin{proof}
By {\minballn}, and discreteness of $\Gamma$, $X$ contains at least one ball of radius $-1$. After removing from $X$ a single ball of radius $-1$ it still contains infinitely many $0$-balls so we can find in $X$ a second ball of radius $-1$. Continuing in this manner we find infinitely many balls in $X$ of radius $-1$.

It follows that if $X$ contains infinitely many balls of radius $0$ then it also contains infinitely many balls of radius $-1$. Repeating the process we obtain in $X$ infinitely many balls of radius $-k$, for every $k\in \mathbb N$.

If $K$ is $\aleph_0$-saturated then the existence of a ball of radius $\gamma$ with $\gamma<-k$ for all $k\in \mathbb{N}$ follows. 
\end{proof}

\begin{lemma}\label{L:no Einf in KO}
Let $\CK=(K,v,\dots)$ be an $\aleph_0$-saturated dp-minimal valued field with $\bk$ finite. If $\CK$ satisfies {\minballn}  then for every infinite definable  $X\subseteq K/\CO$ there exists an infinite definable family of subsets of $X$ containing arbitrarily large finite sets.

\end{lemma}
\begin{proof}
First, recall that since $\CK$ is dp-minimal and $\bk$ is finite we know that $\Gamma$ is discrete. Let $\pi:K\to K/\CO$ be the natural projection and let $Y=\pi^{-1}(X)$. Since $X$ is infinite, $Y$ contains infinitely many closed balls of radius $0$ and by Lemma \ref{L:infinite ball in discrete} and saturation, there exists $a\in Y$ such that $B_{-k}(a)\subseteq Y$ for all $k\in\mathbb{N}$. Since $\Gamma$ is discrete and $\bk$ is finite, for each $k\in \mathbb{N}$, $ B_{-k}(a)$ contains only finitely many, say $n_k$, closed balls of radius $0$, but by the choice of $a$ we get that $\sup_{k\in \mathbb{N}} n_k=\omega$.  The definable sets $\pi(B_{-k}(a))\subseteq X$, $k\in \mathbb N$,  satisfy the requirements.
\end{proof}

\begin{corollary}\label{P-minimal-no K/O} 
Let $\CK=(K,v,\dots)$ be a P-minimal valued field and  $\CF$  an infinite  field interpretable in $\CK$. Then $\CF$ is not locally strongly internal to $K/\CO$. Moreover, there is no definable finite-to-finite correspondence, with bounded fibers, from an infinite subset of $\CF$ into $K/\CO$.  
\end{corollary}
\begin{proof}
We may assume that $\CK$ is $\aleph_0$-saturated.
Assume towards a contradiction that $\CF$ is strongly internal to $K/\CO$. By Proposition \ref{P:dugald}, $\CK$ satisfies {\minballn}, so by Proposition \ref{L:no Einf in KO}, there is a definable family of subsets of $\CF$ containing arbitrarily large finite sets. However, $\CF$ has finite dp-rank, so it eliminates $\exists^{\infty}$ by \cite[Lemma 2.2]{DoGoStrong}. Contradiction.

If there were a definable finite-to-finite correspondece $C\sub  \CF\times K/\CO$ with $T:=\pi_2(C)$ (the projection of $C$ into $K/\CO$) infinite then for any finite $L\sub T$ also $\pi_2^{-1}(L)\sub \CF$ is finite, and by saturation there is $m\in \Nn$ such that $|\pi_2^{-1}(L)|\le m|L|$. We can now, using Lemma  \ref{L:no Einf in KO}, reach the same contradiction as in the previous paragraph.  
\end{proof}

\subsection{Eliminating the sort $\Gamma$}

\begin{proposition}\label{no Gamma} 
Assume that $T$ is either V-minimal, power bounded T-convex, or P-minimal and let $\CF$ be an interpretable field in $\CK=(K,v,\dots)\models T$. Then $\CF$ is not locally strongly internal to $\Gamma$. Moreover, in the P-minimal case there is no definable finite-to-finite definable correspondence, with bounded fibers, from an infinite subset of $\CF$ into $\Gamma$.
\end{proposition}
\begin{proof}
If $T$ is either V-minimal or power bounded T-convex then the induced structure on $\Gamma$ is that of an ordered vector space.
Indeed, for V-minimal $\CK$ this follows from the requirement that the structure induced on RV is the one induced from the pure valued field language, combined with quantifier elimination for ACVF (\cite{holly}).  For T-convex $\CK$ this is \cite[Theorem B]{vdDries-Tconvex}. 

By quantifier elimination for ordered  vector spaces (over an ordered field) it follows that in both of the above cases every definable function in $\Gamma$  is piecewise affine \cite[Chapter 1, Corollary 7.8]{vdDries}. The result follows by Proposition \ref{C:not-internal-to locally affine}.

For the P-minimal case,  every subset of $\Gamma^n$ is definable in the Presburger language $(\Gamma,+,-,0,<,P_n)$. Let $X\subseteq \Gamma$ be an infinite definable subset.
By quantifier elimination we may assume $X$ has the  form $a\leq x\leq b\wedge P_n(x)$, for some $n$, where $a$ and $b$ are not in the same Archimedean component. The family $\{a'<x<b'\wedge P_n(x): a<a'<b'<b\}$, is an infinite definable family of subsets of $X$ containing arbitrarily large finite sets. We may now conclude as in Corollary \ref{P-minimal-no K/O}.
\end{proof}

\section{Classifying interpretable fields}

In Section \ref{S:def fields} we characterised, under various assumptions, fields \emph{definable} in dp-minimal fields (of characteristic 0). We now combine  these results,  as well as the tools developed in the last two sections, to prove our main theorem.

\begin{theorem}\label{T: main thoerem}
	Let $\CK=(K,v, \dots)$ be a dp-minimal valued field and let $\CF$ be an infinite field interpretable in $\CK$. Then: 
	\begin{enumerate}
		\item If $\CK$ is P-minimal and satisfies {\gendif}, then $\CF$ is definably isomorphic to a finite extension of $K$. 
		\item If $\CK$ is a power bounded $T$-convex field then $\CF$ is definably isomorphic to one of $K$, $K(\sqrt{-1})$, $\textbf{k}$ or $\textbf{k}(\sqrt{-1})$. 
		\item If $\CK$ is $V$-minimal then $\CF$ is definably isomorphic to $K$ or $\textbf{k}$. 
	\end{enumerate}
\end{theorem}
\begin{proof}
We may assume that $\CK$ is sufficiently saturated and let $\CF$ be an interpretable infinite field. 

(1) In order to show that $\CF$ is locally strongly internal to one of the distinguished sorts, we use Proposition \ref{new-main} (indeed $\CK$ satisfies the requirement by Proposition \ref{P:dugald}). Consequently there exists an infinite definable  $S\subseteq \CF$ and a definable finite-to-finite correspondence between $S$ and an infinite subset of one of  $K$, $K/\CO$, $\Gamma$ or $\bk$. Since $S$ is infinite, we can eliminate the case of $\bk$. The sorts $K/\CO$ and $\Gamma$  are eliminated by Corollary \ref{P-minimal-no K/O} and Proposition \ref{no Gamma}, respectively.
    
    It follows that  $\CF$ is locally strongly internal to  $K$ and thus,  by Corollary \ref{C:finite-to-finite enough for interp}, $\CF$ is definably isomorphic to a finite extension of $K$.
    
(2+3) By Proposition \ref{P:s-i to sorts, c-min and v-min} (and Remark \ref{R:t-convex sat properties} for the T-convex case), $\CF$ is locally strongly internal to either $K$, $K/\CO$, $\Gamma$ or $\bk$. The cases of $K/\CO$ and $\Gamma$ are eliminated by Corollary \ref{P-minimal-no K/O} and Proposition \ref{no Gamma}, respectively. 

In the T-convex case, the field $K$ is an SW-uniform (valued) field and $\bk$ is o-minimal and hence an SW-uniform field as well.  Thus, by  Theorem \ref{prop-main}, $\CF$ is definably isomorphic to a finite extension of $K$ or $\bk$.

In the V-minimal case, if $\CF$ is locally strongly internal to the SW-uniform field $K$ then Theorem \ref{prop-main} implies that it is definably isomorphic to a finite extension of $K$, so to $K$ itself (as it is algebraically closed).

If $\CF$ is locally strongly internal 
$\bk$ then, since $\bk$ is a pure algebraically closed field, we conclude by Proposition \ref{P:SI to SM}  that it is definably isomorphic to $\bk$.
\end{proof}

\begin{remark}
 %\begin{enumerate}
     As any $p$-adically closed field satisfies {\gendif} (it is enough to check it for finite extensions of the p-adics, where we may apply \cite[Theorem 1.1 and Section 5]{SccdDries}), the above answers Pillay's question on fields interpretable in $\qp$.
   % \item  In the V-minimal case, the proof of the last theorem would work, more generally, whenever $K$ is C-minimal and 1-h-minimal, provided that condition {\BCen} is satisfied. In that case we can not conclude that, if  $\CF$ is locally strongly internal to $\bk$, it is definably isomorphic to $\bk$, but only that $\CF$ is definably isomorphic to a field interpretable in $\bk$. Similarly, it could be that $K$ is definably isomorphic that $\CF$ is definably isomorphic to a real closed field definable in $\Gamma$, or to the algerbaic closure of such a field. 
 %\item In \cite{hensel-minII} the theory of Hensel minimal fields of mixed characteristic is developed. It seems that the analogues of our assumptions {\gendif} 
 %\end{enumerate}
\end{remark}

As a corollary we obtain, using the work of Hempel and Palac\'in \cite{Hempel-Palacin}, a theorem about definable division rings.
\begin{corollary}
Let $\CK=(K,v,\dots)$ be a a valued field and $D$ an infinite interpretable division ring.
\begin{enumerate}
    \item If $\CK$ is a power-bounded T-convex valued field then $D$ is definably isomorphic $K$, $K(\sqrt{-1})$, or the quaternions over $K$, or to  $\bk$,  $\bk(\sqrt{-1}))$, or the quaternions over $\bk$.
    \item If $\CK$ is V-minimal then  $D$ is definably isomorphic to either $K$ or $\bk$.
\end{enumerate}
\end{corollary}
\begin{proof}
As $D$ is of finite dp-rank, by \cite[Theorem 2.9]{Hempel-Palacin}, $D$ is a finite extension of its center, an interpretable field $\CF$. The result follows.
\end{proof}

\subsection{Concluding remarks}
We conclude the paper with a few remarks on the scope of our main results. If an o-minimal field is not power-bounded, by Miller, \cite{miller-dichotomoty},   it is necessarily exponential, i.e.  defines a homomorphism $\exp$ from the additive group of the fields to the multiplicative group of positive elements.

\begin{lemma}
    Let $\CK=(K,v)$ be an exponential $T$-convex valued field. Then there exists an infinite field, $\CF$, interpretable in $\CK$ that is not definably isomorphic to  either one of $K$, $K(\sqrt{-1})$, $\bk$ or $\bk(\sqrt{-1})$. 
\end{lemma}
\begin{proof} Fix $a\in K$ with $v(a)<0$. % We use here $|x|$ to denote the usual absolute value in the sense of the ordered field $K$. 
    Consider $\CO_a:=\exp(a\CO)-\exp(a\CO)$. 
    %\{x\in K: (\exists b\in a\CO) (|x|<  \exp(ab))\}$
     Since $a\CO$ is convex and $\exp$ is monotone the set  $\exp(a\CO)$ is a convex subset of the positive elements in $K$, and therefore $\CO_a$ is also convex. Since $\exp(a\CO)=\exp(a)+\exp(\CO)=\exp(a)+\CO\subsetneq K^{>0}$, it follows that  $\CO_a\ne K$. Since $a\CO$ is an additive subgroup, $\CO_a$ is closed under multiplication. This is obvious for elements of $\exp(a\CO)$ amd the general case follows from convexity of $\CO_a$.   Convexity and closure under multiplication imply that $\CO_a$ is an additive subgroup.   So $\CO_a$ is a convex subring of $K$ containing $\CO$,  hence it is a valuation subring of $K$. Let $\m_a\sub \m$ be its maximal ideal and $\bk_a$ the associated residue field.

    The field $\bk_a$ is real closed (as the residue field of a real closed valued field). It is not isomorphic to $\bk$, because the latter is o-minimal whereas $\bk_a$ is not. Indeed, the image of $\CO$ under the residue map is a convex subring of $\bk_a$. Also, $\bk_a$ is not definably isomorphic to $K$. Indeed, suppose on the contrary that there was a definable injection, $\psi$, from $\bk_a$ into $K$. Since $\CK$ has definable Skolem functions (\cite[Remark 2.7]{vdDries-Tconvex}), such a function $\psi$ would imply the existence of a definable function $\Psi: K\to \CO_a$ such that $\Psi(x)\in \psi(x)/\m_a$ for all $x\in K$. By assumption, if  $x\neq y$ then $\psi(x)/\m_a\neq \psi(y)/\m_a$. So the image of $\Psi$ is discrete, contradicting the weak o-minimality of $\CK$. 
\end{proof}

We do not know whether there are any fields interpretable in $\CK$ other than $K$ itself, the residue fields associated with definable valuation rings and their algebraic closure.  

\appendix 
\section{}\label{A:itay-proof}
We now prove Lemma \ref{L:itaywom}. First we remind the statement: 

\textbf{Lemma}
\emph{    Let $\CM$ be a structure of finite dp-rank and $\mathbb{U}\succ \CM$ a monster model.
    \begin{enumerate}
    \item Let $D$ be an SW-uniformity in $\CM$ and let $b_1,\dots,b_n$ be some tuples in $\mathbb{U}$. For every $\CM$-definable $X$, there exists $a\in X$, with $\dpr(a/M)=\dpr(X)$, such that $\dpr(ab_i/M)=\dpr(a/M)+\dpr(b_i/M)$ for all $1\leq i\leq n$.
    \item For $A\sub \mathbb U$ and $a\in \CM^n$,  there exists a small model $\CN\prec \CM$,  $A\subseteq N$, such that $\dpr(a/A)=\dpr(a/N)$.
\end{enumerate}
}
\begin{proof}
    (1) We proceed by induction on $k=\dpr(X)$. Assume that $k=\dpr(X)=1$.
    
    For any $1\leq i\leq n$, if we set $\dpr(b_i/A)=n_i$ then there are mutually indiscernible sequences over $M$, $\la  I_{i,j}: 0\leq j\leq n_i\ra$, such that none of the $I_{i,j}$ are indiscernible over $Mb_i$.

Since $X$ is infinite and $\CM$ is a model there exists a global type $p\in S(\mathbb{U})$ concentrated on $X$ and invariant over $\CM$. Indeed, take any non-algebraic type concentrated on $X$. It is finitely satisfiable over $\CM$ so we may choose any global type extending it that is still finitely satisfiable over $\CM$ (so invariant over $\CM$). In particular $p$ is $\CN$-invariant for any small $\CN\succ \CM$. Let $\CN\succ \CM$ be a model  containing all the $I_{i,j}$ and let $J=\la a_i:i<\omega \ra$ be a sequence satisfying $a_i\models p|Na_{<i}$, for all $i<\omega$.

Consequently, for all $1\leq i\leq n$, $\{I_{i,0},\dots, I_{i,n_i},J\}$ are $M$-mutually indiscernible but each one is not $Mb_ia_0$ indiscernible. As a result $\dpr(a_0b_i/A)\geq n_i+1$. On the other hand, by subadditivity, $\dpr(a_0b_i/A)\leq \dpr(a_0/A)+\dpr(b_i/A)=1+\dpr(b_i/A)=1+n_i$.

Now, let $k>1$.  By Remark \ref{R:product of one-dim in SW}, we may assume that $X=X_1\times\dots\times X_k$ with $k=\dpr(X)$ and $\dpr(X_i)=1$ for all $1\leq i\leq k$. By the induction hypothesis we can find $a'=(a_2,\dots,a_k) \in X_2\times\dots\times X_k$ such that for all $1\leq i\leq n$, $\dpr(a'b_i/A)=(k-1)+\dpr(b_i/A)$. Now let $b_i'=(a',b_i)$, and using the case $k=1$ we find $a_1\in X_1$ such that for all $1\leq i\leq n$, $\dpr(a_1b_i'/A)=1+\dpr(b_i'/A)=k+\dpr(b_i/A)$.
    
    (2) Let $\la I_t:t<\kappa\ra$ be mutually indiscernible sequences over $A$ witnessing that $\dpr(a/A)\geq \kappa$, i.e. each $I_t$ is not indiscernible over $Aa$, and let $\CM'$ be some small model with $A\subseteq M'$. By \cite[Lemma 4.2]{SiBook}, there exists a mutually indiscernible sequence $\la J_t:t<\kappa\ra$ over $M'$ such that $\tp(I_t:t<\kappa/A)=\tp(J_t:t<\kappa/A)$. Let $\sigma$ be automorphism of $\mathbb{U}$ fixing $A$ and mapping the $J_t$'s to the $I_t$'s. Thus $\la I_t:t<\kappa\ra$ are mutually indiscernible over $N:=\sigma(M')$ and each one is still not indiscernible over $Aa$ so not over $Na$ as well.
\end{proof}

\bibliographystyle{plain}
\bibliography{Bibfiles/harvard}
\end{document}